\documentclass[11pt]{amsart}

\usepackage{setspace} 
\usepackage{amssymb}
\usepackage[dvips]{color}
\usepackage{amsmath}
\usepackage{amsthm}
\usepackage[a4paper]{geometry}
\usepackage{amsfonts}
\usepackage[all]{xypic} 
\usepackage{bbm}
\usepackage{xcolor}
\usepackage[normalem]{ulem}

\usepackage{color}
\usepackage[colorlinks=true,linkcolor=blue,citecolor=blue,pdfpagelabels=false]{hyperref}
\usepackage{cleveref}
\usepackage{url}

\usepackage{hyperref}

\usepackage[colorinlistoftodos]{todonotes}

\allowdisplaybreaks[4] 
\newtheoremstyle{myremark}     {10pt}{10pt}{}{}{\bfseries}{.}{.5em}{}

\newtheorem{thm}{Theorem}[section]
\newtheorem{cor}[thm]{Corollary}
\newtheorem{lem}[thm]{Lemma}
\newtheorem{prop}[thm]{Proposition}
\theoremstyle{definition}
\newtheorem{defn}[thm]{Definition}

\theoremstyle{myremark}
\newtheorem{rem}[thm]{Remark}
\numberwithin{equation}{section}

\def\A{\mathbb A}

\def\N{\mathbb N}

\newcommand{\CB}{\mathcal{B}}
\newcommand{\CC}{\mathcal{C}}

\newcommand{\CM}{\mathcal{M}}
\newcommand{\CN}{\mathcal{N}}

\newcommand{\CP}{\mathcal{P}}

\newcommand{\CW}{\mathcal{W}}

\newcommand{\eps}{\varepsilon}
\newcommand{\ups}{\upsilon}

\newcommand{\abs}[1]{\left\vert#1\right\vert}

\newcommand{\norm}[1]{\left\Vert#1\right\Vert}

\newcommand\inner[2]{\left\langle #1, #2 \right\rangle}

\begin{document}
	
	\title[Maximal Ergodic Inequalities]{ Non-Commutative Maximal Inequalities  for State Preserving Actions of amenable groups }
	
	%    Information for  author
	\author[P. Bikram]{Panchugopal Bikram}
	\address{School of Mathematical Sciences,
National Institute of Science Education and Research,  Bhubaneswar, An OCC of Homi Bhabha National Institute,  Jatni- 752050, India}
	\email{bikram@niser.ac.in}

\author[Hariharan G]{Hariharan G}
	\address{School of Mathematical Sciences,
National Institute of Science Education and Research,  Bhubaneswar, An OCC of Homi Bhabha National Institute,  Jatni- 752050, India}
	\email{hariharan.g@niser.ac.in}

\author[S. Kundu]{Sudipta Kundu}
	\address{School of Mathematical Sciences,
National Institute of Science Education and Research,  Bhubaneswar, An OCC of Homi Bhabha National Institute,  Jatni- 752050, India}
	\email{sudipta.kundu@niser.ac.in}
    
	%    Current address
	%\curraddr{...}
	%\thanks{..}

	%    Information for author
	\author[D. Saha]{Diptesh Saha}
	%    Address of record for the research reported here
	\address{Institute of Mathematics of the Polish Academy of Sciences,  ul. Sniadeckich 8, 
		00–656 Warszawa, Poland}
	%    Current address
	%\curraddr{...}
	\email{dptshs@gmail.com}
	%    \thanks will become a 1st page footnote.
	%\thanks{}

	%    Current address
	%\curraddr{...}
	%\thanks{..}
	
	%    Information for second author
	%\author[D. Saha]{Diptesh Kumar Saha}
	%    Address of record for the research reported here
	%\address{School of Mathematical Sciences,National Institute of Science Education and Research, HBNI, Bhubaneswar, Odisha - 752050, India}
	%    Current address
	%\curraddr{...}
	%\email{dptshs@gmail.com}
	%    \thanks will become a 1st page footnote.
	%\thanks{D. Saha is supported by the Council of Scientific and Industrial Research grant no. 09/1002(0029)/2017-EMR-I, Ministry of Human Resource Development, Government of India.}

	%	    General info
	
	%\date{January 1, 2001 and, in revised form, June 22, 2001.}
	
	%\dedicatory{This paper is dedicated to our advisors.}

	\begin{abstract}
		In this article, we establish maximal inequalities and deduce ergodic theorems for state-preserving actions of amenable, locally compact, second-countable groups on tracial non-commutative $L^1$-spaces. As a further consequence, in combination with the Neveu decomposition, we obtain a stochastic ergodic theorem for amenable group actions.

	\end{abstract}

	\maketitle 
	%	\tableofcontents

	\section{Introduction}
  
In this work, we investigate non-commutative maximal ergodic inequalities and pointwise ergodic theorems for actions of amenable groups that preserve a given state. The study of ergodic theorems in classical measure-theoretic settings originates in the 1930s with the seminal contributions of Birkhoff and von Neumann. Since then, the theory has undergone extensive development. The non-commutative counterparts of these results were initiated in the pioneering work of Lance (\cite{Lance1976}), who established a pointwise ergodic theorem for averages associated with a single state-preserving automorphism of a von Neumann algebra equipped with a faithful normal state. Subsequent generalizations were provided by Kümmerer, Conze, Dang-Ngoc, and others (cf. \cite{Kuemmerer1978}, \cite{Conze1978}, and the references therein).

Over time, it has become well understood that the canonical route to establishing an ergodic theorem proceeds through the derivation of an appropriate maximal inequality together with a proper Banach principle. The non-commutative framework, however, renders the maximal inequality step particularly intricate and technically demanding.

In two foundational papers \cite{Yeadon1977, Yeadon1980}, Yeadon developed maximal ergodic theorems for actions of positive, sub-tracial, sub-unital maps operating on the preduals of semifinite von Neumann algebras and, as a consequence, established associated pointwise ergodic theorems. This line of inquiry was later advanced by Junge and Xu \cite{Junge2007}, who extended Yeadon’s results by proving non-commutative Dunford–Schwartz maximal inequalities in $L^p$-spaces, employing sophisticated interpolation methods.

Parallel to these developments, substantial progress has been made on ergodic theorems for group actions on classical measure spaces (cf. \cite{anantharaman2010ergodic, calderon1953general}). In particular, relevance to the present discussion is the work of Lindenstrauss \cite{Lindenstrauss2001}, which established pointwise ergodic theorems for tempered Følner sequences, thereby extending Birkhoff’s classical theorem to actions of second-countable amenable groups.  

A significant advancement in the non-commutative theory was made by Hong, Liao, and Wang in \cite{Hong2021}. In their work, the authors established a collection of maximal inequalities in non-commutative \(L^{p}\)-spaces for actions of locally compact group metric measure spaces satisfying a doubling condition, as well as for compactly generated groups exhibiting polynomial growth. These inequalities subsequently enabled them to derive the corresponding pointwise ergodic theorems. The scope of these results was later broadened to encompass general amenable groups in \cite{cadilhac2022noncommutative} and very recently for unimodular, amenable groups in \cite{chakraborty2025ergodic}.

%%%%upto this%%%%

It is worth emphasizing that both of the aforementioned results deal exclusively with actions that preserve a faithful, normal, semifinite trace. In contrast, Junge and Xu \cite{Junge2007} examined maximal inequalities for ergodic averages in Haagerup’s \(L^{p}\)-spaces (\(1 < p < \infty\)) associated with state-preserving actions of the discrete group \(\mathbb{Z}\) and of the semigroup \(\mathbb{R}_{+}\). The corresponding endpoint case \(p = 1\) for such state-preserving actions remained open until the more recent developments in \cite{Bik-Dip-neveu} and \cite{pg-ds-ergo-semi}. \\

In \cite{Bik-Dip-neveu}, the authors established maximal inequalities and pointwise ergodic theorems in non-commutative \(L^{1}\)-spaces for actions of locally compact groups of polynomial growth that preserve a faithful, normal state on a von Neumann algebra. Their approach relies primarily on an inequality from \cite[Proposition 4.8]{Hong2021}, which allows the relevant ergodic averages over the group to be dominated by an ergodic average of an associated Markov operator.

In \cite{pg-ds-ergo-semi}, a comparable maximal inequality was derived for actions of \(\mathbb{Z}^{d}_{+}\) and \(\mathbb{R}^{d}_{+}\), employing a similar strategy. Additionally, in \cite{Bik-Dip-Amenable}, the author proved a maximal inequality for actions of group metric measure spaces satisfying the doubling condition, making use of suitable analogues of Calderón’s transference principle together with maximal inequalities for non-commutative martingales.

In the present article, we establish a similar maximal inequality for the state preserving action of an amenable group as in \cite{Bik-Dip-Amenable}, by exploiting the maximal inequality obtained by  Cadilhac and Wang, in \cite{cadilhac2022noncommutative}, which was obtained in tracial setup.\\

Let \( G \) be an amenable, locally compact, second-countable group equipped with a right-invariant Haar measure \( \mu \) with an admissible F\o lner sequence $(F_n)$. 

Let \( \mathcal{M} \) be a finite von Neumann algebra with a faithful, normal tracial state \( \tau \), and let \( (\mathcal{M}, G, \alpha, \rho) \) be a kernel (cf.  Definition \ref{kernel defn}). One of the main results of this article establishes the following maximal inequality and ergodic theorem.

\begin{thm}\label{main thm}
    Let \( \phi \in \mathcal{M}_* \),  then  consider the averages \( (A_k(\phi))_{k \in \mathbb{N}} \), where  
    \begin{align*}
        A_k(\phi)(x) : = \frac{1}{\mu(F_k)} \int_{F_k} \phi(\alpha_{ s^{-1}}(x)) \, d\mu(s), \quad x \in \mathcal{M}.
    \end{align*}
    Then the following hold:  \\
    \begin{enumerate}\setlength\itemsep{.7em}
      
\item Furthermore, for all $j \in \N$, let $\phi_j \in \CM_{* s}$ satisfy $\norm{\phi_j} \leq \frac{1}{4^{s_j}}$. Fix $ \ups > 0$. Then for all $n \in \N$, there exist a sequence $\left(\gamma_j(n)\right)_{j \in \N}$ of positive real numbers and projections $\{p^{\ups}_{n,j} : j \in \N\}$ in $\CM$ such that:
	\begin{enumerate}\setlength\itemsep{.7em} 
		\item $ \displaystyle \lim_{j \to \infty} \gamma_j(n) = 0$, uniformly in $n \in \N$.  
		\item $\rho(1 - p^{ \ups}_{n,j}) \leq (1 + \ups)\gamma_j(n) + \ups$.  
		\item For all $j \in \N$,
		\begin{align}\label{max-cond-0}
		    A_k(\phi_j)(x) \leq  \frac{c s_j}{2^{s_j}} \rho(x), \quad \text{for all } x \in p^{ \ups}_{n,j} \CM_+ p^{\ups}_{n,j} \text{ and } k \in [n].
		\end{align}
	\end{enumerate}
	where, $[n]=\{1,2...,n\}$. The constant $c$ depends only on the group $G$.

\item 
Additionally,  we assume $G$ is unimodular. Then for any $\phi \in \mathcal{M}_*$, there exists an element $\bar{\phi} \in \mathcal{M}_*$ such that, for every $\varepsilon > 0$, one can find a projection $e \in \mathcal{M}$ satisfying $\tau(1-e) < \varepsilon$ and
	\begin{align*}
		\lim_{n \to \infty}
		\sup_{\substack{x \in e\mathcal{M}_+e \\ x \neq 0}}
		\frac{\abs{(A_n(\phi) - \bar{\phi})(x)}}{\tau(x)} = 0 .
	\end{align*}

    \end{enumerate}
\end{thm}
We note that an similar maximal inequality obtained in \cite[Corollary~3.18]{Bik-Dip-Amenable}, which was proved for state-preserving actions on group metric measure spaces. In that setting, the authors derive the estimate
\[
\lvert A_k(\phi_j)(x) \rvert 
\leq C \delta \, \|\phi_j\|\, \|x\|
+ \frac{C}{2^j}\, \rho(x),
\]
for all $x \in q^{\delta,\ups}_{n,j}\, \CM_+\, q^{\delta,\ups}_{n,j}$ and all $k \in [n]$.

In contrast, in the  present maximal inequality, specifically in  \cref{max-cond-0}, the additional error term $C\delta \|\phi_j\|\,\|x\|$ does not appear. Instead, we obtain the sharper bound
\[
A_k(\phi_j)(x) \leq \frac{c s_j}{2^{s_j}}\, \rho(x),
\qquad
\text{for all } x \in q^{\ups}_{j,n}\, \CM_+\, q^{\ups}_{j,n}
\text{ and } k \in [n],
\]
where the auxiliary sequence $(s_j)$ arises naturally in the passage from the trace to the state.\\

  We conclude this introduction by outlining the structure of the paper. In the following section, we review basic notions from the theory of von Neumann algebras and their associated linear functionals. We also introduce noncommutative $L^1$-spaces, discuss the bilateral almost uniform (b.a.u.) topology, and conclude the section with the definition of noncommutative dynamical systems and the corresponding predual actions.

Section~\ref{Maximal Inequality} is devoted to the proof of part~(1) of Theorem~\ref{main thm}. In Section~\ref{Pointwise ergodic theorems}, we establish a pointwise ergodic theorem for ball averages along an admissible F\o lner  sequence, thereby proving part~(2) of Theorem~\ref{main thm}. In addition, we derive an appropriate Banach principle adapted to this setting.

Finally, in the last section, we combine the results obtained in Section~\ref{Pointwise ergodic theorems} with the Neveu decomposition theorem from \cite{Bik-Dip-neveu} to obtain a stochastic ergodic theorem within this framework.

%%%%%%%%%%%%%%%%%%%%%%%%%%%%%%%%%%%%%%%%%%%%%%%%%%%%%%%%%%%%%%%%%%%%%%%%%%%%%%%%%%%%%%%%%%%%%%%%%%%%%%%%%%%%%%%%

	\section{\bf Preliminaries}
Throughout this article, let \(\mathcal{M}\) denote a von Neumann algebra acting on a separable Hilbert space \(\mathcal{H}\). The operator norm inherited from \(\mathcal{B}(\mathcal{H})\) is written as \(\lVert \cdot \rVert_\infty\); when no confusion can arise, we abbreviate it simply by \(\lVert \cdot \rVert\). The commutant of \(\mathcal{M}\) is denoted by \(\mathcal{M}'\). We further write \(\mathcal{P}(\mathcal{M})\) for the collection of projections contained in \(\mathcal{M}\).\\

The predual of \(\mathcal{M}\), is denoted by \(\mathcal{M}_*\), is the closed linear subspace of the Banach dual \(\mathcal{M}^*\) consisting exactly of the ultraweakly continuous linear functionals on \(\mathcal{M}\). A functional \(\phi \in \mathcal{M}^*\) is called positive if \(\phi(x^{*}x) \geq 0\) for every \(x \in \mathcal{M}\), and it is said to be self-adjoint when \(\phi(x) = \overline{\phi(x^{*})}\) holds for all \(x \in \mathcal{M}\). We denote by \(\mathcal{M}_{*+}\) and \(\mathcal{M}_{*s}\) the cones of positive and self-adjoint elements of \(\mathcal{M}_*\), respectively.\\

\subsection{\textbf{Non-commutative $L^1$-spaces}}

%%%

Let \(\mathcal{M} \subseteq \mathcal{B}(\mathcal{H})\) be a von Neumann algebra equipped with a faithful, normal, semifinite (f.n.s.) trace \(\tau\). A (possibly unbounded) operator 
\[
X : \mathcal{D}(X) \subseteq \mathcal{H} \to \mathcal{H},
\]
which is densely defined, closed, and affiliated with \(\mathcal{M}\), i.e,  satisfying \(u'^{*} X u' = X\) for every unitary \(u' \in \mathcal{M}'\)—will be written as \(X \,\eta\, \mathcal{M}\). Such an operator \(X\) is said to be \(\tau\)-measurable if, for every \(\varepsilon > 0\), there exists a projection \(e \in \mathcal{P}(\mathcal{M})\) such that \(\tau(1-e) < \varepsilon\) and \(e\mathcal{H} \subseteq \mathcal{D}(X)\). The collection of all \(\tau\)-measurable operators affiliated with \(\mathcal{M}\) is then defined by
\[
L^{0}(\mathcal{M},\tau) := \{\, X \,\eta\, \mathcal{M} : X \text{ is } \tau\text{-measurable}\,\}.
\]

The space \(L^{0}(\mathcal{M},\tau)\) forms a \(^*\)-algebra under the adjoint operation, strong sum \(X+Y := \overline{X+Y}\), and strong product \(X \cdot Y := \overline{XY}\), where \(\overline{X}\) denotes the closure of \(X\). Detailed treatments may be found in \cite{stratila2019lectures, hiai2021lectures}.

Several natural topologies arise on \(L^{0}(\mathcal{M},\tau)\). The first relevant one is the \emph{measure topology}, whose neighborhoods are of the form
\[
X + \mathcal{N}(\varepsilon,\delta),
\qquad \varepsilon,\delta > 0,\; X \in L^{0}(\mathcal{M},\tau),
\]
where
\[
\mathcal{N}(\varepsilon,\delta)
 := \{\, X \in L^{0}(\mathcal{M},\tau) : \exists\, e \in \mathcal{P}(\mathcal{M}) 
 \text{ with } \tau(1-e) < \delta \text{ and } \|eXe\| < \varepsilon \,\}.
\]
A net \(\{X_i\}_{i\in I}\) converges to \(X\) in the measure topology precisely when, for every \(\varepsilon,\delta>0\), there exists \(i_0\in I\) such that for all \(i\ge i_0\) one may find \(e_i\in \mathcal{P}(\mathcal{M})\) with \(\tau(1-e_i) < \delta\) and
\[
\|\, e_i (X_i - X)e_i \,\| < \varepsilon.
\]

By \cite[Theorem 1]{nelson1974notes}, the space \(L^{0}(\mathcal{M},\tau)\) is complete in the measure topology. Furthermore, this topology is metrizable and renders \(\mathcal{M}\) dense in \(L^{0}(\mathcal{M},\tau)\).

The second topology of interest is the \emph{bilateral almost uniform} (b.a.u.) topology.

\begin{defn}\label{bau conv defn}
A net \(\{X_i\}_{i\in I}\subseteq L^{0}(\mathcal{M},\tau)\) is said to converge \emph{bilaterally almost uniformly} (b.a.u.) to \(X\in L^{0}(\mathcal{M},\tau)\) if, for every \(\varepsilon>0\), there exists a projection \(e \in \mathcal{P}(\mathcal{M})\) such that \(\tau(1-e) < \varepsilon\) and 
\[
\lim_{i} \|\, e(X_i - X)e \,\| = 0.
\]
\end{defn}

The following fundamental fact, due to \cite[Theorem 2.2]{litvinov2024notes}, will be used subsequently.

\begin{thm}\label{complete wrt bau}
The space \(L^{0}(\mathcal{M},\tau)\) is complete with respect to the b.a.u.\ topology.
\end{thm}

We record a simple but useful observation.

\begin{prop}
Let \(\{X_i\}_{i\in I}\) and \(\{Y_i\}_{i\in I}\) be nets in \(L^{0}(\mathcal{M},\tau)\) such that \(X_i \to X\) and \(Y_i \to Y\) in measure (resp.\ b.a.u.). Then, for every scalar \(c\in \mathbb{C}\), the net \(cX_i + Y_i\) converges to \(cX+Y\) in measure (resp.\ b.a.u.).
\end{prop}

The trace \(\tau\) on \(\mathcal{M}_+\) extends in the standard way to \(L^{0}(\mathcal{M},\tau)_+\) by
\[
\tau(X) := \int_0^{\infty} \lambda \, d\tau(e_\lambda),
\]
where \(X = \int_{0}^{\infty} \lambda\, de_\lambda\) is the spectral decomposition of \(X\). The associated noncommutative \(L^1\)-space is defined by
\[
L^{1}(\mathcal{M},\tau)
 := \{\, X \in L^{0}(\mathcal{M},\tau) : \tau(|X|) < \infty \,\},
\]
with norm \(\|X\|_1 := \tau(|X|)\). The following properties, proved in \cite[Chapter 4]{hiai2021lectures}, will be used repeatedly.

\begin{thm}\label{predual prop}
\begin{enumerate}
    \item \(L^{1}(\mathcal{M},\tau)\) is a Banach space under \(\|\cdot\|_1\).
    \item If \(X \in L^{1}(\mathcal{M},\tau)\) and \(y \in \mathcal{M}\), then both \(Xy\) and \(yX\) lie in \(L^{1}(\mathcal{M},\tau)\), and moreover \(\tau(Xy)=\tau(yX)\).
    \item The canonical identification
    \[
    \phi \in \mathcal{M}_* \quad \longleftrightarrow \quad 
    X \in L^{1}(\mathcal{M},\tau), \quad 
    \phi(x) = \tau(Xx) \;\; (x\in \mathcal{M})
    \]
    defines a surjective linear isometry from \(L^{1}(\mathcal{M},\tau)\) onto \(\mathcal{M}_*\), which preserves positivity.
\end{enumerate}
\end{thm}

%%%
\subsection{Non-commutative dynamical systems}
Let $G$ be a locally compact, second countable, Hausdorff group equipped with a
right-invariant Haar measure $\mu$.
A \emph{Følner sequence} $(F_n)_{n \in \mathbb{N}}$ is a sequence of measurable subsets
of $G$ satisfying
\[
0 < \mu(F_n) < \infty \quad \text{for all } n \in \mathbb{N},
\]
and such that for every $g \in G$,
\[
\lim_{n \to \infty}
\frac{\mu\bigl(F_n g \,\triangle\, F_n\bigr)}{\mu(F_n)} = 0,
\]
where $\triangle$ denotes the symmetric difference.

Throughout this article,  we further, assume that $G$ is an amenable group with a Følner sequence
$(F_n)$   in $G$. In addition, for most of our results,
we shall assume that $(F_n)$ is an \emph{admissible Følner sequence}
(see \cite{cadilhac2022noncommutative} for the precise definition).\\

\noindent We also remark that if $G$ is a locally compact, second countable, amenable group,
then every Følner sequence admits an admissible Følner subsequence; see \cite{cadilhac2022noncommutative} for details.

\begin{defn}\label{nc dyn sys} Let \((E,\|\cdot\|)\) be a real ordered Banach space.
Then a \emph{non-commutative dynamical system} is a triple \((E,G,\gamma)\), where \(\gamma : G \to \mathcal{B}(E)\) is a mapping satisfying \(\gamma_s \circ \gamma_t = \gamma_{st}\) for all \(s,t \in G\), and the following conditions hold:
\begin{enumerate}
    \item for every \(a \in E\), the orbit map \(s \mapsto \gamma_s(a)\) is continuous on \(G\), where \(E\) is equipped with the norm topology (and the \(w^*\)-topology in the case \(E=\mathcal{M}\));
    \item the operators are uniformly bounded, i.e., \(\sup_{s \in G}\|\gamma_s\| < \infty\);
    \item positivity is preserved: if \(a \ge 0\), then \(\gamma_s(a) \ge 0\) for all \(s\in G\).
    
\end{enumerate}
\end{defn}
\noindent We further assume that if $E = \CM$, $\gamma_{s}$ is a *-automorphism on $\CM$ for all $s \in G$.

\noindent Let \(\mathcal{M}\) be a von Neumann algebra with a faithful, normal, semifinite trace \(\tau\), and consider the dynamical system \((L^{1}(\mathcal{M},\tau), G, \gamma)\). Since \(\mathcal{M} = (L^{1}(\mathcal{M},\tau))^{*}\), the Banach-space adjoint of \(\gamma_s\), denoted \(\gamma_s^{*}\), is characterized by
\begin{align}\label{dual trans}
    \tau(\gamma_s^{*}(x)\, Y) = \tau(x\, \gamma_s(Y)), 
    \qquad x \in \mathcal{M}, \; Y \in L^1(\mathcal{M},\tau).
\end{align}

Similarly, for a non-commutative dynamical system \((\mathcal{M}, G, \beta)\), the predual operator of \(\beta_s\), denoted \(\widehat{\beta_s}\), is defined via
\begin{align}\label{predual trans}
    \tau(\beta_s(x)\, Y) = \tau(x\, \widehat{\beta_s}(Y)),
    \qquad x \in \mathcal{M}, \; Y \in L^{1}(\mathcal{M},\tau).
\end{align}
In particular, one has \((\widehat{\beta_s})^{*} = \beta_s\) for every \(s \in G\).

Now let \(T : \mathcal{M} \to \mathcal{M}\) be a bounded normal operator. Its adjoint \(T^{*} : \mathcal{M}^{*} \to \mathcal{M}^{*}\) is given by
\[
    T^{*}(\phi)(x) = \phi(T(x)), \qquad x \in \mathcal{M},\; \phi \in \mathcal{M}^{*}.
\]
If \(\phi \in \mathcal{M}_{*}\subseteq \mathcal{M}^{*}\), then \(T^{*}(\phi)\) also belongs to \(\mathcal{M}_{*}\); hence the restriction \(T^{*}|_{\mathcal{M}_{*}}\colon \mathcal{M}_{*}\to \mathcal{M}_{*}\) serves as the predual transformation associated with \(T\).

To streamline notation, we shall not distinguish between \(T\) and its predual \(T^{*}|_{\mathcal{M}_*}\). Thus, when \(\phi \in \mathcal{M}_{*}\) and we write \(T(\phi)\), it is to be understood that \(T\) acts as the predual map. Under this convention,
\begin{align}\label{predualmap}
    T(\phi)(x) = \phi(T(x)),
    \qquad x \in \mathcal{M},\; \phi \in \mathcal{M}_{*}.
\end{align}

%%%%
%%%%%
%%%%%

%%%%%%%%%%%%%%%%%%%%%%%%%%%%%%%%%%%%%%%%%%%%%%%%%%%%%%%%%%%%%%%%%%%%%%maximal%inequality%%%%%%%%%%%%%%%%%%%%%%%%%%%%%%%%%%%%%%%%%%%%%%%%%%%%%%%%%%%%%%%%%%%%%%%%%%%%%%

\section{\textbf{Maximal Inequality}}\label{Maximal Inequality}
In this section we establish a maximal inequality associated to a state preserving action of an amenable group  on von Neumann algebra.

Suppose  $ \CM$ be a von Neumann algebra with a f.n.s. trace $\tau$. Then, consider the von Neumann algebra $\CN =: L^\infty(G, \mu) \otimes \CM$  and consider the f.n.s trace $\widetilde{\tau}$ on $\CN$, defined by 
$$ \widetilde{\tau}= \int_G( \cdot ) d\mu  ~\otimes \tau $$

We note that $ L^1(\CN, \widetilde{\tau})$ is identified with 
$ L^1\left(G, L^1(\CM, \tau) \right)$. Further, we note that  
$ L^1(\CM, \tau)$ can be identified with $\CM_*$ via the following map: 
$$ L^1(\CM, \tau) \ni A \to \tau_A \in \CM_*, $$
where $\tau_A(x) = \tau(Ax) $ for $ x \in \CM$. This will establish an isometric  isomorphism between $ L^1(\CM, \tau)$ and  $\CM_*$. Consequently, $ L^1\left(G, L^1(\CM, \tau) \right)$ is isometrically isomorphic to  $ L^1(G, \CM_*)$  via the following map 
$$ L^1\left(G, L^1(\CM, \tau) \right) \ni f \mapsto 
\tau_f \in L^1(G; \CM_*), $$
where $ \tau_f(s) =: \tau_{ f(s)}$. 
Further, as  $\CN= L^\infty(G, \mu) \otimes \CM \cong $ $L^\infty(G; \CM):=$ the space of weakly measurable essentially bounded functions from $G$ to $\CM$, we obtain that 
$$ \widetilde{\tau}_f(g) =:\widetilde{ \tau}(f g ) = \int_G \tau( f(s) g(s) ) d\mu(s),$$
where $f \in L^1\left( G, L^1(\CM, \tau) \right)$   and   $g \in   L^\infty(G; \CM)$. We also note that $ \widetilde{\tau}( g) = \int_G \tau( g(s)) d\mu(s), \text{ for } g \in L^\infty(G; \CM)$.
Let $ (F_n)$ be a F\o lner sequence in $G$ and $ f \in L^1( G) \otimes M_{ *}$
Then consider 
$$ \A_n(f)(s) = \frac{ 1}{ \mu( F_n)}\int_{ F_n} f(st) d\mu(t).$$

\begin{lem}\label{average-relation}
Let $ f \in L^1\left( G,  \CM_*\right)_+$ and 
$ h \in L^1( G; L^1(\CM, \tau) )_+ $ such that 
$$ f =\widetilde{ \tau}_h, $$ then it follows that 
$$ \A_n(f) = \widetilde{ \tau }_{ \A_n(h)}, ~ \text{ for all } n \in \N.$$
\end{lem}
\begin{proof}
Let $ g \in L^\infty(G; \CM)$  and   $ t \in G$, then observe that $ \rho_t(g) \in L^\infty(G, \CM) $, where $ \rho_t(g)(s) = g(st^{-1}) $ for all $ s \in G$. 
As $ f =\widetilde{ \tau}_h $, so, we have 
$$ \int_G f(s)(\rho_t(g)(s)) d\mu = \int_G \tau( h(s) \rho_t(g)(s) ) d\mu.$$
Further, since $\mu$ is right-invariant Haar measure, we have
$$ \int_G f(st)(g(s)) d\mu(s) = \int_G \tau( h(st) (g(s) ) d\mu(s)$$

Thus, we observe that 
\begin{align*}
    \int_G \A_n(f)(s) ( g(s) )  d\mu(s) 
    &=  \int_G  \frac{1}{\mu(F_n)}\int_{ F_n} f(st) ( g(s) )  d\mu(t) d\mu(s)\\
    &= \frac{1}{\mu(F_n)}\int_{ F_n}  \int_G  f(st) ( g(s) )   d\mu(s)   d\mu(t)\\
&= \frac{1}{\mu(F_n)}\int_{ F_n}  \int_G  \tau(h(st) ( g(s) )  )  d\mu(s)   d\mu(t)\\
&=  \int_G  \frac{1}{\mu(F_n)}\int_{ F_n}  \tau(h(st) ( g(s) )  )  d\mu(t)  d\mu(s)  \\
 &=  \int_G \tau(    \A_n(h)(s) g(s) ) d\mu(s).   \end{align*}
This completes the proof.

\end{proof}

\noindent Then by linearity of the operator $\A_n$ and the weight $\tilde{\tau}$, we observe that for all $\nu \in L^1\left( G,  \CM_*\right)$ there exists $ B_\nu \in L^1\left(G; L^1(\CM,\tau) \right)$ such that 
$$ \nu( \cdot) = \widetilde{\tau}( B_\nu (\cdot) ), \text{ i.e}, ~\nu = \widetilde{ \tau}_{ B_\nu }. $$ We can also have the following observation.

% \begin{lem}\label{norm-1}
% Let $ \nu \in L^1\left( G,  \CM_*\right)$, then we observe that 
% there exists $ B_\nu \in L^1\left(G; L^1(\CM,\tau) \right)$ such that 
% $$ \nu( \cdot) = \widetilde{\tau}( B_\nu (\cdot) ), \text{ i.e}, ~\nu = \widetilde{ \tau}_{ B_\nu }. $$
% Suppose   $ e $  a projection in $\CN$ and $ \delta >0$, then $ \norm{e(B_\nu)e} \leq \delta $ if and only if $$\abs{(\nu)(exe)} \leq \delta \widetilde{\tau}(exe),  ~ \text{ for all } x \in \CN_+.$$
%    \end{lem}

\begin{lem}\label{norm}
% Let $ \nu \in L^1\left( G,  \CM_*\right)$, then we observe that 
% there exists $ B_\nu \in L^1\left(G; L^1(\CM,\tau) \right)$ such that 
% $$ \nu( \cdot) = \widetilde{\tau}( B_\nu (\cdot) ). $$
Let   $ e $  be a  projection in $\CN$ and $ \delta >0$, then $ \norm{e(B_\nu)e} \leq \delta $ if and only if $$\abs{(\nu)(exe)} \leq \delta \widetilde{\tau}(exe),  ~ \text{ for all } x \in \CN_+.$$
   \end{lem}
\begin{proof}
Suppose 
$$\abs{(\nu)(exe)} \leq \delta \widetilde{\tau}(exe),  ~ \text{ for all } x \in \CN_+,$$
then for all $ x \in \CN_+$, it follows that  
\begin{align*}
&\abs{(\nu)(exe)} \leq \delta \widetilde{\tau}(exe)\\
\Leftrightarrow & \abs{ \widetilde{\tau}((B_\nu )(exe))} \leq \delta \widetilde{\tau}(exe) \\
\Leftrightarrow & \abs{ \widetilde{\tau}(e(B_\nu)ex)} \leq \delta \widetilde{\tau}(exe)
 \end{align*}
which is equivalent to $\norm{e(B_\nu)e} \leq \delta$.

\end{proof}

\noindent We recall the following theorem from \cite{cadilhac2022noncommutative}
for smooth reference. 
\begin{thm}\label{amenable-max}
Let $(F_n)$ be an admissible F\o lner sequence in $G$. Let $f \in L^1\left(G, L^1(\CM,  \tau)\right)$ and $\eps > 0$, then there exists a projection $e^\eps$ in $\CN$ such that 
\begin{enumerate}
        \item $ \tau( 1-e^\eps) \leq \frac{\norm{ f}_1}{ \eps}$ and 
        \item $ \norm{ e^\eps \A_k(f) e^\eps }\leq c \eps, ~ \text{ for all } k \in \N.$
    \end{enumerate}
\end{thm}

\noindent Now using the \Cref{norm} we immediately deduce the following theorem from \Cref{amenable-max}. This version will be used in our analysis. 

\begin{thm}\label{maxthm}
   Let $K$ be a Borel set in $G$ and  $ f \in L^1(G; \CM_*)$ with $\text{supp}(f) \subseteq K$.  Then for given  $\eps >0$, there exists a projection with $e_K^\eps $ in $\CN$ with $\text{supp}(e_K^\eps) \in L^\infty( K,\mu) \otimes \CM$  satisfying the following 
    \begin{enumerate}\setlength\itemsep{.7em} 
        \item $ \widetilde{ \tau }( 1_K\otimes 1-e_K^\eps) \leq \frac{\norm{ f}_1}{ \eps}$ and 
        \item $ \norm{ \A_k(f)( e^\eps_K  x e^\eps_K  )} \leq c \eps \widetilde{\tau}( e^\eps_K  x e^\eps_K ) $,  for all $x \in \CN_+$ and   $k \in \N$.
    \end{enumerate}
\end{thm}
\begin{proof}
Note that  $ f \in L^1(G; \CM_*)_+$, so, $ f : G \to \CM_*$ is a measurable function and 
$$\norm{ f}_1 = \int_G \norm {f(s)} d\mu.$$
Further, note that there exists a $h \in L^1(G; L^1(\CM, \tau))_+ $ such that 
$$ \int_G f(s) g(s) d\mu   = \int_G \tau( h(s)g(s)  ) d\mu,  \text{ for all } g \in L^\infty( G; \CM),$$
 i.e, $f = \widetilde{\tau}_h $.  Further, it follows that
 $ \norm{f}_1= \int_G \tau( h(s)) d\mu.$

\noindent From the Theorem \ref{amenable-max} recall that, for given $h \in L^1(G; L^1( \CM, \tau) )$  and $ \eps >0$, there exists a projection $ e^\eps$ in  $\CN=L^\infty(G; \CM )$ such that   
\begin{enumerate}
        \item $ \widetilde{\tau}( 1-e^\eps) \leq \frac{\norm{ f}_1}{ \eps}$ and 
        \item $ \norm{ e^\eps \A_k(h) e^\eps }\leq c \eps, ~ \text{ forall } k \in \N.$
    \end{enumerate}
Therefore, as $ \A_k(f) = \widetilde{\tau}_{ \A_k(h)}$ (\Cref{average-relation}), so,  from the \Cref{norm}, it follows that 
\begin{align}\label{max-ine-Eq}
  \A_k(f) ( e^\eps x e^\eps)  \leq c \eps \widetilde{ \tau}(( e^\eps x e^\eps)  , ~ \text{ for all }  x \in \CN_+ \text{ and } k \in \N.
\end{align}
Further,  consider $$e_K^\eps = (1_K \otimes 1) e^\eps,$$
but for notational convenient we write
$ e_K^\eps = 1_K  e^\eps$.  As $ \CN = L^\infty(G, \mu) \otimes \CM$, we note that 
$$(1_K \otimes 1) e^\eps= e^\eps (1_K \otimes 1) e^\eps.$$
Thus, it follows that $e^\eps_K$ is a projection in $\CN$.  
Now since $ e^\eps_K \leq e^\eps$, so, the last part of the theorem follows from the   \Cref{max-ine-Eq}.

\end{proof}

\noindent For an admissible F\o lner sequence $(F_n)$ of $G$,  
we consider  the following averages: 
\begin{align*}
	A_n(\phi)(x)= \frac{1}{\mu(F_n)}\int_{ F_n} \phi(\alpha_{s^{-1}}(x)) d\mu(s), \quad \phi \in \CM_*, \quad x \in \CM, \quad n \in \N,
\end{align*}

\begin{rem}\label{comp supp meas}
    For each $n \in \N$, we define probability measures on $G$ by $\mu_n(E)= \frac{\mu(E \cap F_{n})}{\mu(F_{n})}$, where $E$ is a measurable subset of $G$. Then observe that 
    \begin{align*}
& A_n(\phi)(x)= \int_{G}  \phi(\alpha_{s^{-1}}(x)) d\mu_n(s), \quad \phi \in \CM_*, ~x \in \CM, 
\text{ and } \\
    &\A_nf(s)= \int_{G} f(st) d\mu_n(t), \quad f \in \CN_*,~ s \in G.
\end{align*}
Moreover, we observe that $G$ has an increasing sequence of compact sets $( K_n )$ such that   such that for all $n \in \N$, $\underset{ n \to \infty }{ \lim} \mu_n(K_n)= \mu_n(G)$. 
Now for $m \in \N$, we consider 
 \[
    A_{n,m}(\phi)(x):= \int_{G}  \phi(\alpha_{s^{-1}}(x)) \chi_{K_m}(s) d\mu_n(s), \quad \phi \in \CM_*, ~x \in \CM,
    \]
and we notice that 
    \[
    A_{n,m}(\phi) \leq A_n(\phi), \text{ and } \lim_m A_{n,m}(\phi)(x)= A_n(\phi)(x) \text{ for all } \phi \in \CM_{*+}, ~x \in \CM_+.
    \]
    Similar observations can also be drawn for $(\A_n)$. Therefore, without loss of generality, we may assume that the support of $\mu_n$'s are compact.  In fact, we invoke this intermediate assumption only in the proof of \Cref{maximal ineq for averages} to employ the associated transference principle.

\end{rem}

Consider the following set 
	\begin{align}\label{compacttset}
	    K_n := \bigcup_{k=1}^{n} F_k \cup \{ 1^G\}, 
		\end{align}
where $1^G$ is the identity element of $G$ and note that  $K_n$	is a compact subset of $G$.

Now we have the following proposition which establishes a relation between the average of the action $A_n$ and the average of the function space. It can be proved for both right and left Haar measure. As a proof  is known and straightforward, we omit the proof.
\begin{prop} Let $ F\subseteq G$ be any compact set,  $ \phi \in \CM_* $ and $f \in L^1(G, \CM_*)$ defined by  $ f(t)=  \chi_{ F K_n}(t) 
 \alpha_{t^{-1}}(\phi)$. Then it follows that 
  $$
\alpha_{s^{-1}} A_k(\phi) = \A_k f(s) \quad \text{for all } s \in F \text{ and } k \in [n]. 
    $$
where, $[n]$ denotes the set $\{1,2,...,n\} \subset \mathbb{N}$.
\end{prop}               

\begin{rem}\label{trace-proj}
Let $\CM$ be a von Neumann algebra equipped with a f.n tracial state $\tau$, and a f.n state  $ \rho$. 
Then we note that there exists a unique  $A \in L^1(M, \tau)_+$ such that $ \rho( \cdot ) = \tau( A( \cdot)).$

\noindent We recall that $ \CN = L^\infty(G, \mu) \otimes \CM$ and the f.n.s trace $ \widetilde{\tau} $ defined as  $ \widetilde{\tau} = \int_G( \cdot) \otimes \tau$. Further, we consider the f.n semi-finite weight on $ \CN_+$, defined as  $ \widetilde{ \rho} = \int_G( \cdot) \otimes \rho $. 

\noindent Now suppose $X= 1 \otimes A$ and then note that   $\widetilde{\rho}(x) = \widetilde{\tau}(Xx)$ for all $x \in \CN$.

\noindent For any $s > 0$, consider the projection $q_s :=  \chi_{(1/s, s)}(A) \in \CM$. Observe that $\tau(1 - q_s) \xrightarrow{s \to \infty} 0$. Therefore, for every $\eps > 0$, there exists $s_\eps > 0$ such that $\tau(1 - q_{s_\eps}) < \eps$. Furthermore, it follows that $A q_{s_\eps} \geq  \frac{1}{s_\eps} q_{s_\eps}$. 

Now let $ E $ be a Borel set in $ G$ with $\mu(E) < \infty $. Then consider $ r_s = 1_E \otimes q_s$, then we have $$\frac{\widetilde{\tau}(1_E \otimes 1 - r_s)}{\mu(E)} = \tau(1-q_s) \xrightarrow{s \to \infty} 0.$$
Further, we have $ \frac{\widetilde{\tau}(1_E \otimes 1 - r_{s_\eps})}{\mu(E)} < \eps $. As $ X = 1\otimes A$, we also note that $ X r_{ s_\eps} \geq \frac{1}{s_\eps} r_{s_\eps} $.\\
Thus, for all $0 \neq x \in r_{s_\eps} \CN_+ r_{s_\eps}$, we have
		\begin{align*}
			\frac{\widetilde{\rho}(x)}{\widetilde{\tau}(x)} = \frac{\widetilde{\tau}(Xx)}{\widetilde{\tau}(x)} &= \frac{\widetilde{\tau}(X _{s_\eps} x)}{\widetilde{\tau}(x)} \quad  \text{(since $r_{s_\eps} x = x$)}\\
			&\geq \frac{1}{s_\eps}  \frac{\widetilde{\tau}( x)}{\widetilde{\tau}(x)} = \frac{1}{s_\eps}.
		\end{align*}
Thus, it follows that 
$$ \frac{\widetilde{\tau}( x) }{\widetilde{\rho}(x)} \leq s_\eps, \text{ for all } x \in r_{s_\eps} \CN_+ r_{s_\eps} $$
Now for $j \in \N$,  suppose $ \eps = \frac{1}{4^j}$ and   write $ s_{\frac{1}{4^j}} = s_j $,  ~$ r_j =: r_{ s_j} $. For all $ j \in \N$,  we obtain that 
\begin{align}\label{trace-state}
\frac{\tau( 1_E \otimes 1-r_j)}{\mu(E)} < \frac{1}{4^j} \text{ and } \frac{\widetilde{\tau}( x) }{\widetilde{\rho}(x)} \leq s_j, \text{ for all } x \in r_{j} \CN_+ r_{j}.
\end{align}

\noindent We note that $ s_j \to \infty $ as $j \to \infty$. Furthermore, observe that the sequence $(s_j)$ only depends on $ (\frac{1}{4^j})$ and the Radon Nikodym derivative operator $A$ of $\rho$ with respect to the trace $\tau$.
\end{rem}

\noindent We list the following lemma and omit the proof, as it is straightforward. 
\begin{lem}\label{point-eva}
    Let $ f \in L^1(G, \CM_*)_+  $, $ e$ be a projection in $ \CN = L^\infty(G; \CM) $ and $ \delta > 0$, such that 
    $$  f( e x e) \leq \delta \widetilde{\rho}( exe), \text{ for  all } x \in \CN_+. $$
    Then it follows that 
    $$ f(s)(e(s) x e(s) ) \leq \delta \rho( e(s) x e(s) ), ~ \text{ for all } x \in \CM_+$$
    and $\mu$-a.e $s \in G$.
\end{lem}

\noindent 
Let $\CM$ be a von Neumann algebra equipped with an f.n.s. weight $\rho$. We denote the associated GNS Hilbert space by $L^2(\CM, \rho)$. The algebra 
$\CM$ admits a concrete representation as a subalgebra of 
$\CB(L^2(\CM, \rho))$. For any 
$ x \in \CM$ with  $ \rho( x^*x) < \infty $,
we denote by 
$ \hat{x}$ the corresponding GNS vector in 
$L^2(\CM, \rho)$. The inner product on the GNS space will be denoted by 
$\inner{\cdot}{\cdot}_\rho$, and the associated norm by $\norm{\cdot}_\rho$.
\noindent
Furthermore, if 
$\rho$ is a state, the cyclic and separating vector associated with 
$\rho$  will be denoted by 
$\Omega_\rho$. 

Now we account the following lemma for smooth references, a proof can be found in \cite{Bik-Dip-Amenable}.

\begin{lem}\label{uniform conv lem}
Let \( \CM_1 \) and \( \CM_2 \) be two von Neumann algebras. Let \( \varphi \) be an f.n.s. weight on \( \CM_1 \), and let \( \rho \) and \( \tau \) be an f.n. state and an f.n. tracial state on \( \CM_2 \), respectively.
Suppose that
\[
\CM_1 \,\otimes\, \CM_2 \subseteq \mathcal{B}\big(L^2(\CM_1 \,\otimes\, \CM_2,\, \varphi \otimes \rho)\big).
\]
Further, let $S$ be an arbitrary set, and $
\{ p_j{(s)} : j\in \mathbb{N},~ s \in S  \} \subseteq \mathcal{P}(\CM_1)$ and  $
\{ x_j{(s)} : j\in \mathbb{N},~ s \in S \} \subseteq \CM_1 \,\otimes\, \CM_2
$ be a  family  of projections and a 
 family of operators, respectively, such that 
$$ 0< \varphi(p_j(s)) < \infty \text{ and } 0\leq x_j(s) \leq p_j(s) \otimes 1, ~ \text{ for all } j \in \mathbb{N}, ~ s \in S.
$$
Define
\[
f_j(s) := \frac{p_j(s) \otimes 1 - x_j(s)}{\varphi(p_j(s))}, \quad \text{for } j \in \mathbb{N}, ~s\in S.
\]
Then the following statements are equivalent:
\begin{itemize}\setlength\itemsep{.7em}
    \item[(1)] \( \displaystyle \lim_{j \to \infty} (\varphi \otimes \rho)(f_j(s)) = 0 \), uniformly in \( s \in S \),
    \item[(2)] \( \displaystyle \lim_{j \to \infty} (\varphi \otimes \tau)(f_j(s)) = 0 \), uniformly in \( s \in S \).
\end{itemize}
\end{lem}

Let $\rho$ be a f.n. state on $\CM$, and let $\tau$ be a f.n. tracial state on $\CM$. We recall the following functionals on $L^\infty(G, \mu) \otimes \CM$:
\[
\tilde{\rho}( \cdot ) = \left(\int_{ G} \otimes \rho \right)(\cdot), 
\quad \text{and} \quad 
\tilde{\tau}( \cdot ) = \left(\int_{ G} \otimes \tau \right)(\cdot).
\]
We observe that $\tilde{\rho}$ is a 
f.n.s.  weight, while $\tilde{\tau}$ is a f.n.s. tracial weight on $L^\infty(G, \mu) \otimes \CM$.

\begin{defn}\label{kernel defn}
Let $G$ be amenable group and 	 and $(\CM, G, \alpha)$ be a non-commutative dynamical system (cf. Definition \ref{nc dyn sys}) such that $\rho \circ \alpha_s = \rho$ for all $s \in G$, where $\rho$ is a f.n. state on $\CM$. We call the quadruple $(\CM, G, \alpha, \rho)$ a kernel. 
\end{defn}
\begin{thm}\label{maximal ineq for averages}
	Let $\CM$ be a finite von Neumann algebra with a faithful normal tracial state $\tau$, and let $(\CM, G, \alpha, \rho)$ be a kernel. For all $j \in \N$, let $\phi_j \in \CM_{*+}$ satisfy $\norm{\phi_j} \leq \frac{1}{4^{s_j}}$. Fix $\ups > 0$. Then for all $n \in \N$, there exist a sequence $\left(\gamma_j(n)\right)_{j \in \N}$ of positive real numbers and projections $\{p^{\ups}_{n,j} : j \in \N\}$ in $\CM$ such that:
	\begin{enumerate}\setlength\itemsep{.7em} 
		\item $ \displaystyle \lim_{j \to \infty} \gamma_j(n) = 0$, uniformly in $n \in \N$.  
		\item $\rho(1 - p^{ \ups}_{n,j}) \leq (1 + \ups)\gamma_j(n) + \ups$.  
		\item For all $j \in \N$,
		\begin{align*}
		    A_k(\phi_j)(x) \leq  \frac{c s_j}{2^{s_j}} \rho(x), \quad \text{for all } x \in p^{ \ups}_{j, n} \CM_+ p^{\ups}_{j, n} \text{ and } k \in [n].
\end{align*}
	\end{enumerate}
	The constant $c$ depends only on the group $G$.
\end{thm}

\begin{proof}
	Let $n \in \mathbb{N}$, $\delta, \upsilon > 0$, and let the sequence $(\phi_j)_{j \in \mathbb{N}} \subset \mathcal{M}_{*+}$ with $\|\phi_j\| \leq \frac{1}{4^j}$ for all $j \in \mathbb{N}$ be fixed. Conisider the $K_n$ as in Eq. \ref{compacttset}, i.e, 
	\[
	K_n := \bigcup_{k=1}^{n} F_k \cup \{ 1^G\},
	\]
	is a compact subset of $G$. Then by the amenability of $G$, choose a compact set $F$ such that 
	\[
	\frac{\mu(FK_n)}{\mu(F)} < 1+ \upsilon.
	\]
 Note that $\mu(FK_n) < \infty$ and $F \subseteq FK_n$ as $K_n$ contains the identity element of the group $G$. Consider the  compact set $F \subseteq G$ and define, for all $j \in \mathbb{N}$, 
	\[
	f_j(t) := \chi_{FK_n}(t) \alpha_{t^{-1}}(\phi_j).
	\]
	Observe that $f_j \in \mathcal{N}_{*+}$ and that $\|f_j\| \leq \mu(KF) \|\phi_j\|$. Furthermore, we note that 
	\[
	\alpha_{s^{-1}} A_n(\phi_j) = \A_n f_j(s) \quad \text{for all } s \in F.
	\]

For  $ \eps = \frac{1}{ 2^{s_j}}$ and for $f_j \in L^1(\CN, \tau)$,  by \Cref{maxthm},  there exists a projection $ e_{j,n} $ in $\CN$ such that 
\begin{enumerate}\setlength\itemsep{.7em} 
        \item $ \widetilde{ \tau }( \chi_{FK_n}\otimes 1-e_{j, n}) \leq 2^{s_j}\norm{ f_j}_1 \leq \frac{\mu( FK_n)}{2^{s_j}} $ and 
        \item $   \A_k(f_j)(  e_{j, n} x e_{j, n} ) \leq \frac{c }{2^{s_j}} \widetilde{\tau}( e_{j, n} x e_{j, n}), ~k \in [n]$, $ x \in \CN_+$ and  a.e in $\mu$.
\end{enumerate}
Now suppose $q_{j, n} = e_{j, n} \wedge r_j $ for $ j \in \N$, where $(r_j)$ is the sequence of projections as discussed as in the Remark \ref{trace-proj}. Then note that
\begin{align*}
 \widetilde{ \tau }( \chi_{FK_n}\otimes 1-q_{j, n})    
&\leq  \widetilde{ \tau }( \chi_{FK_n}\otimes 1-e_{j, n}) + \widetilde{ \tau }( \chi_{FK_n}\otimes 1-r_j)\\
&\leq \frac{\mu(FK_n)}{2^{s_j}}  + \frac{\mu(FK_n)}{4^j}= \mu(FK_n) \left( \frac{1 }{2^{s_j}} + \frac{1}{4^j}\right) 
\end{align*}
Let 
$ x \in q_{j, n} \CN_+ q_{j, n}, ~j \in \N \text{ and } k \in [n]$, then note that 
\begin{align*}
\A_k(f_j)( x ) = \A_k(f_j) \left( q_{j, n} x q_{j, n} \right) &\leq \frac{c }{2^{s_j}} \tau( e_{j, n} x e_{j, n})\\
&\leq \frac{c }{2^{s_j}} \widetilde{\tau}( x)\\
&= \frac{c }{2^{s_j}}  \widetilde{\rho}( x) \frac{\widetilde{\tau}( x)}{\widetilde{\rho}( x)} \\ 
&\leq  \frac{c s_j }{2^{s_j}}\widetilde{\rho}( x)~  (\text{see Eq.}  \ref{trace-state}),
\end{align*}
Therefore, further from  the \Cref{point-eva}, it follows that 
$$ \A_k(f_j)(s) ( q_{j,n }(s) x q_{j,n }(s) )
\leq \frac{c s_j }{2^{s_j}} \rho( q_{j,n }(s) x q_{j,n }(s)), ~ \text{ for all } x \in \CM_+ \text{ and } k \in [n],$$
a.e $\mu$ in  $s \in FK_n$. 
As $ F \subseteq FK_n$, there exists a Borel set $F_\mu$ with $\mu( F\setminus F_\mu) = 0$ and for all $ s \in F_\mu$, we have 
\begin{align}
\A_k(f_j)(s) ( q_{j,n }(s) x q_{j,n }(s) )
\leq \frac{c s_j }{2^{s_j}} \rho( q_{j,n }(s) x q_{j,n }(s)), ~ \text{ for all } x \in \CM_+ \text{ and } k \in [n],    
\end{align}

\noindent
For all  $n \in \mathbb{N}$, we further observe that 
\begin{align*} 
	\frac{\widetilde{\tau}(1-q_{j,n})}{\mu(FK_n)} \leq \left( \frac{1 }{2^{s_j}} + \frac{1}{4^j}\right) \xrightarrow{ j \to \infty } 0.
\end{align*}
Hence,
\begin{equation*}
	\lim_{j \to \infty} \frac{\widetilde{\tau}(1-q_{j,n})}{\mu(FK_n)}= 0, \quad \text{uniformly in }  n \in \mathbb{N}.
\end{equation*}
Consequently, by \Cref{uniform conv lem}, we obtain
\begin{equation*}
	\lim_{j \to \infty} \frac{\widetilde{\rho}(1-q_{j,n})}{\mu(FK_n)} =0, \quad \text{uniformly in }  n \in \mathbb{N}.
\end{equation*}
Suppose 
\[
\gamma_j(n) = \frac{1}{\mu(FK_n)} \widetilde{\rho} \left( 1-q_{j, n} \right),
\]
and we have 
\[
\lim_{ j \to \infty } \gamma_j(n) = 0,
\]
uniformly in $n \in \N$.

\noindent
Now, we find a \( t_\ups \in F \) such that 
\begin{align*}
	\rho  
	\left( 
	1- \alpha_{t_\ups^{-1}} 
	\left( 
	q_{j,n} 
	\left(t_\ups \right)
	\right) 
	\right) 
	&= \rho \left(1-  q_{j, n} \left(t_\ups \right) 
	\right)  \quad \text{since } \alpha_{t_\ups^{-1}} \left(1 \right) =1 \text{ and } \rho \circ \alpha_{t} = \rho.\\
	&\leq \inf_{s \in F_{\mu}} \rho \left(
	1-  q_{j, n} 
	\left( s \right) 
	\right) + \ups.
\end{align*}
We denote $ p_{j,n}^{v} = \alpha_{t_v^{-1}}(q_{j,n}(t_\ups)).$
Then, we note that 
\begin{align*}
	\rho(1- p_{j, n}^{\ups}) 
	&\leq \frac{1}{\mu(F_{\mu})} \int_{F_{\mu}}  \rho \left(1- q_{j, n}(s) \right) d\mu(s) + \ups\\
	&\leq \frac{\mu(FK_n)}{\mu(F)} \frac{1}{\mu(FK_n)} \tilde{\rho} \left( 1-q_{j, n} \right) + \ups \\
	&\leq ( 1+ \ups)  \gamma_j(n) + \ups.
\end{align*} 
Furthermore,  for all \( x \in  \CM_+ \) and \( k \in [n] \), we observe that
\begin{align*}
	A_k(\phi_j) (  p_{j, n}^{\ups} x p_{j, n}^{ \ups}) 	
	&= A_k(\phi_j) 
    \left( \alpha_{t_\ups^{-1}} \left( q_{j, n}(t_\ups) \right) \right) x  \left( \alpha_{t_\ups^{-1}} \left( q_{j, n}(t_\ups) \right)\right) \\
        &= \alpha_{t_\ups^{-1}} A_k(\phi_j) \left(  q_{j, n}(t_\ups)  \alpha_{t_\ups} (x)   q_{j, n}(t_\ups) \right) \\
&= \A_k(f_j)(t_\ups) \left(  q_{j, n}(t_\ups)  \alpha_{t_\ups} (x)   q_{j, n}(t_\ups) \right)\\
&\leq    \frac{cs_j}{2^{s_j}} (\rho\circ \alpha_{t_\ups} ) \left(   \alpha_{t_\ups^{-1}} \left( q_{j, n}(t_\ups) \right)  ~x  ~\alpha_{t_\ups^{-1}}\left(q_{j, n}(t_\ups) \right) \right)\\ 
&\leq    \frac{c   s_j}{2^{s_j}} \rho \left( p_{j, n}^{ \ups} x p_{j, n}^{\ups} \right), (\text{ as } \rho\circ \alpha_{t_\ups} = \rho ).
\end{align*}

\end{proof}

\begin{cor}\label{maximal ineq for averages-cor}

Let $\CM$ be a finite von Neumann algebra with a faithful normal tracial state $\tau$, and let $(\CM, G, \alpha, \rho)$ be a kernel. For all $j \in \N$, let $\phi_j \in \CM_{* s}$ satisfy $\norm{\phi_j} \leq \frac{1}{4^{s_j}}$. Fix $ \ups > 0$. Then for all $n \in \N$, there exist a sequence $\left(\gamma_j(n)\right)_{j \in \N}$ of positive real numbers and projections $\{p^{\ups}_{n,j} : j \in \N\}$ in $\CM$ such that:
	\begin{enumerate}\setlength\itemsep{.7em} 
		\item $ \displaystyle \lim_{j \to \infty} \gamma_j(n) = 0$, uniformly in $n \in \N$.  
		\item $\rho(1 - p^{ \ups}_{n,j}) \leq (1 + \ups)\gamma_j(n) + \ups$.  
		\item For all $j \in \N$,
		\begin{align}\label{max-cond}
		    A_k(\phi_j)(x) \leq  \frac{c s_j}{2^{s_j}} \rho(x), \quad \text{for all } x \in p^{ \ups}_{n,j} \CM_+ p^{\ups}_{n,j} \text{ and } k \in [n].
		\end{align}
	\end{enumerate}
	The constant $c$ depends only on the group $G$.

\end{cor}

\begin{proof}
	Let $\phi_j \in \CM_{*s}$ be such that $\norm{\phi_j} \leq \frac{1}{4^j}$. Note that for all $j \in \N$, there exist $(\phi_j)_1, (\phi_j)_2 \in \CM_{*+}$ such that $\phi_j = (\phi_j)_1 - (\phi_j)_2$ and $\norm{\phi_j} = \norm{(\phi_j)_1} + \norm{(\phi_j)_2}$. 
	
	By applying Theorem \ref{maximal ineq for averages} to $\psi_j = (\phi_j)_1 + (\phi_j)_2$, we obtain a sequence $(\gamma_j(n))_{j \in \N}$ of positive real numbers such that, for all $n \in \N$, there exist projections $\{p^{\ups}_{n,j}: j \in \N\}$ in $\CM$ satisfying:
	\begin{enumerate}\setlength\itemsep{.7em}       
		\item $\displaystyle \lim_{j \to \infty} \gamma_j (n) = 0$ uniformly in $n \in \N$.
		\item The inequality
		\[
		\rho(1 - p^{\ups}_{n,j}) \leq (1 + \ups) \gamma_j(n) + \ups
		\]
		holds.
		\item For all $j \in \N$, we have
		\[
		A_k(\psi_j)(x) \leq \frac{cs_j}{2^{s_j}} \rho(x)
		\]
		for all $x \in p^{\ups}_{n,j} \CM_+ p^{\ups}_{n,j}$.
	\end{enumerate}

	Furthermore, for all $x \in e^{i, \delta, \ups}_{n, \eps} \CM_+ e^{i, \delta, \ups}_{n, \eps}$, and for all $j \in \N$ and $k \in \N$, we have:
	\begin{align*}
		\abs{A_k(\phi_j)(x)}
		&= \abs{A_k((\phi_j)_1)(x) - A_k((\phi_j)_2)(x)}\\
		&\leq A_k(\psi_j)(x)\\
		&\leq \frac{cs_{j}}{2^{s_j}} \rho(x),
	\end{align*}
	
\end{proof}

\begin{rem}\label{rem: max ineq rel}
\begin{enumerate}
    \item Corollary \ref{maximal ineq for averages-cor} is analogous to the maximal inequality established in \cite[Corollary~3.18]{Bik-Dip-Amenable}, which was proved in the setting of state preserving actions on group metric measure spaces. In particular, \cite[Corollary~3.18]{Bik-Dip-Amenable} concludes with the estimate
\[
\lvert A_k(\phi_j)(x) \rvert 
   \leq C\delta \, \| \phi_j \| \, \| x \|
   + \frac{C}{2^j}\, \rho(x),
\]
for all \(x \in q^{\delta,\ups}_{n,j}\, \mathcal{M}_+\, q^{\delta,\ups}_{n,j}\) and all \(k \in [n]\), where the family of projections \((q^{\delta,\ups}_{n,j})\) is defined as in \cite[Corollary~3.18]{Bik-Dip-Amenable}.

In our setting, however, the maximal inequality-specifically the Inequality~\ref{max-cond}, is obtained from the from maximal inequality of  \cite{cadilhac2022noncommutative}. As a consequence, no error term of the form \(C\delta \|\phi_j\|\,\|x\|\) appears (see the Inequality \ref{max-cond}). Instead, we obtain 
\begin{align*}
		    A_k(\phi_j)(x) \leq  \frac{c s_j}{2^{s_j}} \rho(x), \quad \text{for all } x \in q^{ \ups}_{j, n} \CM_+ q^{\ups}_{j, n} \text{ and } k \in [n],
		\end{align*}
where the auxiliary sequence \((s_j)\) arises naturally when passing from the trace to the state.

Thus, to obtain the pointwise ergodic theorem, we need to prove an appropriate Banach principle, which will be established in the next section.

\item

We note that a maximal inequality analogous to \cite[Theorem~4.13]{Bik-Dip-neveu} can be established for ergodic averages taken over a suitably chosen F\o lner sequence in a unimodular amenable group. This follows from the ergodic average dominance result proved  very recently in \cite[Theorem~3.1]{chakraborty2025ergodic}.

We also remark that, when $G$ is assumed to be unimodular, the maximal inequality stated in \Cref{maximal ineq for averages-cor} applies to ergodic averages associated with F\o lner sequences satisfying structural properties different from those considered in \cite[Theorem~3.1]{chakraborty2025ergodic}.

\end{enumerate}
\end{rem}

		\section{\textbf{Pointwise ergodic theorems} }\label{Pointwise ergodic theorems}

		\noindent In this section, we study the pointwise convergence of specific ergodic averages. With the weak-type maximal inequality at our disposal 
		(cf.  \cite{cadilhac2022noncommutative}), 
		it remains to establish an appropriate Banach principle and demonstrate the pointwise convergence of elements from a dense subset. 
		
		\noindent Throughout this section, we assume that $\CM$ is a von Neumann algebra with a faithful normal tracial state $\tau$ and that $G$, with the right Haar measure, is an amenable group with right F\o lner sequence $(F_n)$. Additionally, we assume that $(\CM, G, \alpha)$ is a non-commutative dynamical system for which there exists a faithful normal state $\rho$ on $\CM$ satisfying $\rho \circ \alpha_s = \rho$ for all $s \in G$, making $(\CM, G, \alpha, \rho)$ a kernel.
		
		Before proceeding further, we fix the following notation. Let 
$\{\gamma_s : s \in G\}$ be a family of (appropriate) maps acting on a Banach space $E$. 
For the associated averaging operators, with  a mild abuse  of notations, we write:
\begin{itemize}
    \item $A_n(\cdot) := \frac{1}{\mu(F_n)} \int_{F_n} \gamma_{s^{-1}}(\cdot)\, d\mu(s)$ and \\
 
    \item $ B_n(\cdot) := \frac{1}{\mu(F_n)} \int_{F_n} \gamma_s(\cdot)\, d\mu(s).$
\end{itemize}
Strictly speaking, these operators should be denoted by $A_n^{\gamma}$ and 
$B_n^{\gamma}$ to emphasize their dependence on the family $\{\gamma_s\}_{s \in G}$.
However, for notational simplicity, we suppress this dependence and write $A_n$ and
$B_n$ throughout.

		\noindent Since $\rho \circ \alpha_s = \rho$ for all $s \in G$, we observe that one can define a unitary  $u_s$ on the \textit{G.N.S.} space $L^2(\CM, \rho)$ given by 
		\[
		u_s(x \Omega_\rho) = \alpha_s(x) \Omega_\rho,
		\]
		satisfying $u_s \circ u_t = u_{st}$ for all $s, t \in G$, where $\Omega_\rho$ is the associated cyclic and separating vector in $L^2(\CM, \rho)$. Throughout this section, we denote the Hilbert space inner product and norm in the aforementioned $G.N.S.$ space by $\inner{\cdot}{\cdot}_\rho$ and $\norm{\cdot}_\rho$, respectively. \\
		
		\noindent Now, we present the following mean ergodic theorem.

        		\begin{thm}\label{Mean erg}
			Let ${F_n}$ be F\o lner sequence with all the notation given in the introduction of this section. Then, the averages 
            \begin{align*}
                B_nx := \frac{1}{\mu(F_n)} \int_{F_n} u_s(x) \, d\mu(s), \quad x \in L^2(\CM, \rho),
            \end{align*}
				 
			converge in $\norm{\cdot}_\rho$.
		\end{thm}
		
		\begin{proof}
			Since $(L^2(\CM, \rho), G, u)$ forms a non-commutative dynamical system, we observe that 
			\begin{align*}
				L^2(\CM, \rho) = F_2 \oplus F_2^\perp,
			\end{align*}
			where 
			\begin{align*}
				F_2 &= \{x \in L^2(\CM, \rho) : u_s(x) = x \text{ for all } s \in G\}, \\
				F_2^\perp &= \overline{\text{span}} \{ y - u_s(y) : y \in L^2(\CM, \rho),~ s \in G \}.
			\end{align*}
			
			Therefore, it suffices to check the convergence of $A_{n}x$ for $x = y - u_s(y)$, where $y \in L^2(\CM, \rho)$ and $s \in G$. Moreover, we observe that
			\begin{align*}
				B_nx &= \frac{1}{\mu(F_n)}\int_{F_n}u_t(y)d\mu(t) - \frac{1}{\mu(F_n)}\int_{F_n}u_{ts}(y)d\mu(t)\\
                &=\frac{1}{\mu(F_n)}\int_{F_n}u_t(y)d\mu(t) - \frac{1}{\mu(F_n)}\int_{F_ns}u_{t}(y)d\mu(t)\\
                &=\frac{1}{\mu(F_n)}\int_{G}(\chi_{F_n}-\chi_{F_ns})u_t(y)d\mu(t)
			\end{align*}
			Consequently, we get
			\begin{align*}
				\norm{B_n(x)} &\leq \frac{1}{\mu(F_n)}\int_{G}\abs{\chi_{F_n}-\chi_{F_ns}}\norm{u_t(y)}d\mu(t)\\
                &\leq \frac{1}{\mu(F_n)}\int_{G}\chi_{F_n \Delta F_ns}\norm{y}d\mu(t)\\
                &=\norm{y} \frac{\mu(F_n \Delta F_ns)}{\mu(F_n)}
			\end{align*}
			Since $(F_n)$ is a F\o lner sequence the R.H.S converges to 0 as $n \rightarrow \infty$. Hence $A_n(x)$ converges to 0.
\end{proof}

        		\noindent We recall the following lemma from \cite{Bik-Dip-neveu}.
		
		\begin{lem}[{\cite[Lemma 4.5]{Bik-Dip-neveu}}] \label{key lem1}
			Let $(\CM,G, \alpha)$ be a non-commutative dynamical system with a f.n. state $\rho$, as described at the beginning of this section. Then, there exists a non-commutative dynamical system $(\CM',G, \alpha')$, where $\CM'$ is the commutant of $\CM$, such that 
			\begin{align*}
				\inner{\alpha'_{s}(y')x \Omega_\rho}{\Omega_\rho}_\rho = \inner{y' \alpha_{s^{-1}}(x) \Omega_\rho}{\Omega_\rho}_\rho
			\end{align*}
			for all $x \in \CM$, $y' \in \CM'$, and $s \in G$.
		\end{lem}
		
		\noindent Let $G$ be an amenable group with a F\o lner sequence $(F_n)$ and  $ (M, G, \alpha) $ be dynamical system. Then we recall the following averages; 
		\begin{align*}
			A_{n}(\nu) := \frac{1}{\mu(F_n)} \int_{F_n} \alpha_{ s^{-1}}(\nu) \, d\mu(s), \quad \nu \in \CM_*,
		\end{align*}
		 \noindent We also recall the following norm convergence result from the same paper, with the additional condition that $\rho$ is invariant under the action of $G$.
		
		\begin{thm}[{\cite[Theorem 4.6]{Bik-Dip-neveu}}] \label{mean erg thm for state}
			For all $\nu \in \CM_*$, there exists $\bar{\nu} \in \CM_*$ such that 
			\begin{align*}
				\bar{\nu} = \norm{\cdot}\text{-} \lim_{n \to \infty} A_{n}(\nu).
			\end{align*}
            Further, if the group $G$ is unimodular and that the F{\o}lner sequence $\{F_{n}\}_{n \in \mathbb{N}}$ is chosen to be symmetric, then, $\overline{\nu}$ is $G$-invariant.
		\end{thm}
            \begin{proof}
Let $L^{2}(M,\rho)$ be the closure of $M$ with respect to the norm induced from the
inner product $\langle \cdot , \cdot \rangle_{\rho}$. Then we can define the following contractions on the Hilbert space
$L^{2}(M,\rho)$:
\[
u_{g}(x\Omega_{\rho})=\alpha_{g}(x)\Omega_{\rho}, \qquad x\in M_{s},\ g\in G.
\]

Now, consider
\[
B_{n}:=\frac{1}{m(F_{n})}\int_{F_{n}}u_g\, d\mu(g).
\]
Then by the von Neumann mean ergodic theorem, it follows that for all $\xi\in L^{2}(M,\rho)$, $T_{n}(\xi)$ converges
to $P\xi$ strongly, where $P$ is the orthogonal projection of $L^{2}(M,\rho)$ onto the subspace
\begin{center}
$\{\xi\in L^{2}(M,\rho):u_{g}\xi=\xi\ \text{for all }g\in G\}$.
\end{center}
Let $x\in M$, then it follows that $T_{n}(x\Omega)$ converges in $L^{2}(M,\rho)$.

\noindent Let $y_{1},y_{2}\in M'(\sigma)$ and define $\psi_{y_{1},y_{2}}(x)=\langle xy_{1}\Omega_{\rho},y_{2}\Omega_{\rho}\rangle_{\rho}$
for all $x\in M$. Then by the discussion just before Remark 4.3, there exists a $z\in M$ such that
$y_{1}^{*}y_{2}\Omega_{\rho}=z\Omega_{\rho}$, where $z=J\sigma_{i/2}(y_{2}^{*}y_{1})J$.
Consequently, for $x\in M$,
\[
A_{n}(\psi_{y_{1},y_{2}})(x)
    =\frac{1}{\mu(F_{n})}\int_{F_{n}}
      \langle \alpha_{g^{-1}}(x)y_{1}\Omega_{\rho},y_{2}\Omega_{\rho}\rangle_{\rho}\, d\mu(g)
\]
\[
    =\frac{1}{\mu(F_{n})}\int_{F_{n}}
      \langle \alpha_{g^{-1}}(x)\Omega_{\rho},y_{1}^{*}y_{2}\Omega_{\rho}\rangle_{\rho}\, d\mu(g)
\]
\[
    =\frac{1}{\mu(F_{n})}\int_{F_{n}}
      \langle \alpha_{g^{-1}}(x)\Omega_{\rho},z\Omega_{\rho}\rangle_{\rho}\, d\mu(g).
\]

\[
= \frac{1}{\mu(F_{n})} \int_{F_{n}} \langle x\Omega_{\rho},\, u_{g^{-1}}^{*}(z\Omega_{\rho}) \rangle_{\rho}\, d\mu(g)
\]
\[
= \langle x\Omega_{\rho},\, B_{n}(z\Omega_{\rho}) \rangle_{\rho}\, d\mu(g).
\]

Hence, for all $x\in M$, 
\[
A_{n}(\psi_{y_{1},y_{2}})(x) \to \langle x\Omega_{\rho}, \eta \rangle_{\rho},
\quad\text{where }\eta = \lim_{n\to\infty} A_{n}(z\Omega_{\rho}).
\]
Consider $\overline{\psi}_{y_{1},y_{2}} \in M_{*}$ defined by
\[
\overline{\psi}_{y_{1},y_{2}}(x) = \langle x\Omega_{\rho}, \eta \rangle_{\rho}, \qquad x\in M.
\]

Then by standard argument it follows that 
\[
\overline{\psi}_{y_{1},y_{2}} \circ \alpha_{g} = \overline{\psi}_{y_{1},y_{2}}
\quad\text{for all } g\in G,
\]
and
\[
\|\overline{\psi}_{y_{1},y_{2}}\| = \|\cdot\|_{1}
    - \lim_{n\to\infty} B_{n}(\psi_{y_{1},y_{2}}).
\]

Hence the result follows since the set 
\[
\{\psi_{y_{1},y_{2}} : y_{1},y_{2}\in M'(\sigma)\}
\]
is total in $M_{*}$.

        \end{proof}
        
\begin{rem}\label{invariantnu}
Now, if G is unimodular, we can take $(F_n)_{n \in \mathbb{N}}$ to be symmetric and left and right F\o lner both. In that case, $\overline{\nu} $ is G invariant. Indeed, for $ \nu \in M_*$, recall that 
\[
A_n(\nu)=\frac{1}{\nu(F_n)}\int_{F_n}\nu\circ\alpha_{g^{-1}}\,d\mu(g),
\text{ and }
\overline\nu=\|\cdot\|_1-\lim_{n\to\infty} A_n(\nu).
\]
Fix $h\in G$, we observe  that $\overline\nu\circ\alpha_h=\overline\nu$. Indeed, we notice  that 
\begin{align*}
A_n(\nu)\circ\alpha_h
=\frac{1}{\mu(F_n)}\int_{F_n}\nu\circ\alpha_{g^{-1}}\circ\alpha_h\,d\mu(g)
&=\frac{1}{\mu(F_n)}\int_{F_n}\nu\circ\alpha_{g^{-1}h}\,d\mu(g)\\
&=\frac{1}{\mu(F_n)}\int_{F_n}\nu\circ\alpha_{{(h^{-1}g)}^{-1}}\,d\mu(g)\\
&= \frac{1}{\mu(F_n)}\int_{h^{-1}F_n}\nu\circ\alpha_{g^{-1}}\,d\mu(g) \text{ [ by unimodularity]}
\end{align*}

\noindent Hence the difference is
\[
A_n(\nu)\circ\alpha_h - A_n(\nu)
=\frac{1}{\mu(F_n)}\Big(\int_{h^{-1}F_n}-\int_{F_n}\Big)\nu\circ\alpha_{g^{-1}}\,d\mu(g).
\]
Taking the norm (dual norm on $M_*$) and using $\|\nu\circ\alpha_{g^{-1}}\| \leq \|\nu\|$ we get the estimate
\[
\|A_n(\nu)\circ\alpha_h - A_n(\nu)\|
\le \frac{1}{\mu(F_n)}\int_{F_n\Delta (h^{-1}F_n)} \|\nu\circ\alpha_{g^{-1}}\|\,d\mu(g)
\le \|\nu\|\frac{\mu\big(F_n\Delta (h^{-1}F_n)\big)}{\mu(F_n)}.
\]
Because $\{F_n\}$ is a F{\o}lner sequence and $G$ is unimodular (so left and right Haar measures agree), the ratio
\[
\frac{\mu\big(F_n\Delta (h^{-1}F_n)\big)}{\mu(F_n)}\to 0\qquad(n\to\infty).
\]
Thus $\|A_n(\nu)\circ\alpha_h - A_n(\nu)\|\to 0$ as $n\to\infty$. Passing to the $\|\cdot\|_1$--limit that defines $\overline\nu$ yields
\[
\overline\nu\circ\alpha_h=\overline\nu.
\]
Since $h\in G$ was chosen arbitrarily, $\overline\nu$ is $G$--invariant.

\end{rem}
\noindent We obtain the following convergence result for a noncommutative dynamical system.
        		\begin{lem}\label{dense set-lem}
			Let $(E, G, \gamma)$ be a dynamical system, where $(E, \norm{\cdot})$ is a Banach space. Then, for every $k \in \N$ and $a \in E$, we have
			\begin{align*}
				\lim_{n \to \infty} \norm{B_n (a - B_k(a))} = 0,
			\end{align*}
			where
			\begin{align*}
				B_m(a) := \frac{1}{\mu(F_{m})} \int_{F_{m}} \gamma_s(a) \, d\mu(s), \quad m \in \N.
			\end{align*}
		\end{lem} 

		\begin{proof}
			For all $n \in \N$, we have
			\begin{align*}
				B_n(a - B_k(a)) 
				&= \frac{1}{\mu(F_{n})} \int_{F_{n}} \Big[ \gamma_s(a) - \gamma_s (B_k(a)) \Big] d\mu(s) \\
				&= \frac{1}{\mu(F_{n})} \int_{F_{n}} \Bigg[ \gamma_s (a) - \gamma_s \bigg( \frac{1}{\mu(F_{k})} \int_{F_{k}} \gamma_t (a) d\mu(t) \bigg) \Bigg] d\mu(s) \\
				&= \frac{1}{\mu(F_{n})} \frac{1}{\mu(F_{k})} \int_{F_{n}} \int_{F_{k}}  \Big( \gamma_s (a)- \gamma_{st}(a) \Big) d\mu(t) d\mu(s) \\
				&= \frac{1}{\mu(F_{n})} \frac{1}{\mu(F_{k})} \int_{F_{k}} \int_{F_{n}}  \Big( \gamma_s (a)- \gamma_{st}(a) \Big) d\mu(s) d\mu(t) \quad \text{(by Fubini's theorem)} \\
				&= \frac{1}{\mu(F_{k})} \int_{F_{k}} \frac{1}{\mu(F_{n})} \Bigg( \int_{F_{n}}   \gamma_s (a)d\mu(s)- \int_{F_{n}t} \gamma_s (a)d\mu(s) \Bigg) d\mu(t)\\
                &= \frac{1}{\mu(F_{k})} \int_{F_{k}} \frac{1}{\mu(F_{n})} \Bigg( \int_{G}(\chi_{F_{n}}-\chi_{F_{n}t})(s)\gamma_s (a)d\mu(s) \Bigg) d\mu(t).
			\end{align*}
			
		\noindent We observe that $\abs{\chi_{F_{n}}-\chi_{F_{n}t}}=\chi_{F_{n}\Delta F_{n}t}$\text{.} Since $\sup_{g \in G}\|\gamma_{s}\|$ is finite, we  get 
        \begin{align*}
 \left\| B_n(a - B_k(a)) \right\| &\leq   \frac{1}{\mu(F_k)} \int_{F_k} \Big( \frac{1}{\mu(F_n)} \int_{G} \abs{(\chi_{F_n}- \chi_{F_n t})(s)} \left\|\gamma_s(a)\right\| \, d\mu(s)\Big)d\mu(t)  \\
&\le \frac{ \| a \|\sup_{g \in G}\| \gamma_g \| }{\mu(F_{k})}\int_{F_k}\frac{\mu(F_n \Delta F_n t) }{\mu(F_{n})}d\mu(t)
\end{align*}
Since the functions $g_{n}(t)=\frac{\mu(F_n \Delta F_n t)}{\mu(F_{n})}$ are bounded on $F_{k}$ with $\mu(F_k) < \infty$ for each $k\in \mathbb{N}$, and converges to 0 pointwise (by the definition of amenability), by the Dominated Convergence Theorem,  we obtain for every $k \in \mathbb{N}$ and  $a \in E$,
\begin{align*}
\lim_{n \rightarrow \infty}\left\| B_n(a - B_k(a)) \right\|=0.
\end{align*}
		\end{proof}

% \textcolor{red}{For the non-commutative dynamical system $(\CM, G, \alpha)$ with f.n state $\rho$ satisfying $\rho \circ \alpha_s = \rho$ for all $s \in G$, by Theorem \ref{mean erg thm for state}, it follows that $\norm{\cdot}_1- \lim_{i \to \infty} B_{r_i}(\nu)$ exists for all $\nu \in \CM_*$ and the limit is denoted by $\bar{\nu}$.}

\noindent Given a kernel $(\CM, G, \alpha, \rho)$, we observe that by \Cref{mean erg thm for state}, it follows that $\norm{\cdot}_1- \lim_{n \to \infty} B_{n}(\nu)$ exists for all $\nu \in \CM_*$, and the limit is denoted by $\bar{\nu}$. 

\noindent We recall another result from \cite{Bik-Dip-neveu}, which identifies a dense subset of $\CM_*$ such that for elements of this set, the ergodic averages converge in a manner that is similar to bilateral almost uniform convergence. This result requires that the group $G$ be unimodular, due to our use of remark \ref{invariantnu}.
\begin{prop}[\cite{Bik-Dip-neveu}, Lemma 4.9]\label{dense set}
	Consider the following set:
	\begin{align*}
		\CW_1 := \{ \nu - A_{k}(\nu ) + \bar{\nu} : \quad  k \in \N,\  \nu \in \CM_{*+} \text{ with } \nu \leq \lambda \rho \text{ for some } \lambda >0 \}.
	\end{align*}
	\begin{enumerate}\setlength\itemsep{.7em}
		\item[(i)] Define $\CW = \CW_1 - \CW_1$. Then, $\CW$ is dense in $\CM_{* s}$.
		
		\item[(ii)] For all $\nu \in \CW$, we have  
		\begin{align}\label{maximal type}
			\lim_{n \to \infty} \sup_{x \in \CM_+, x \neq 0} \frac{\abs{(A_{n}(\nu)- \bar{\nu})(x)}}{\rho(x)} =0.
		\end{align}
	\end{enumerate}
\end{prop}
\begin{proof}
	\emph{(i):} The proof follows from a similar argument as in the proof of \cite[Lemma 4.9, Part (i)]{Bik-Dip-neveu}. Let $\mu\in M_{*+}$ and $\varepsilon>0$. By lemma 3.19 from \cite{hiai2021lectures}, we find a $\nu\in M_{*+}$ with $\nu\leq \lambda\rho$ for some $\lambda>0$, such that $\|\mu-\nu\|<\varepsilon/2$. Further, by \ref{mean erg thm for state} we know that $A_n(\nu)$ is convergent, and write $\bar{\nu}=\lim_{n\to\infty} A_n(\nu)$.
So there exists an $n_0\in\mathbb{N}$ such that
\[
\|\bar{\nu}-A_{n_0}(\nu)\|\leq \varepsilon/2.
\]
Therefore by the triangle inequality, we have
\[
\|\mu-(\nu-A_{n_0}(\nu)-\bar{\nu})\|
\leq \|\mu-\nu\|+\|A_{n_0}(\nu)-\bar{\nu}\|
<\varepsilon.
\]

\noindent Now for $\mu\in M_{*s}$, we write $\mu=\mu_{+}-\mu_{-}$, where
$\mu_{+},\mu_{-}$ are normal positive linear functionals. Thus, it follows
that $\mathcal{W}$ is dense in $M_{*s}$.
	
	\emph{(ii):} Fix $k \in \N$ and consider $\nu_k := \nu - A_{k}(\nu ) + \bar{\nu}$. It suffices to prove Eq. \ref{maximal type} for $\nu_k$. 
	
	Now, since 
	\begin{align*}
		A_{n}(\nu_k) - \bar{\nu} = A_{n}(\nu - A_{k}(\nu))
	\end{align*}
	for all $n \in \N$, it follows from Lemma \ref{dense set-lem} that $\overline{\nu_k}= \bar{\nu}$. 
	
	Since $\nu \leq \lambda \rho $, by \cite[Lemma 4.4]{Bik-Dip-neveu}, there exists a unique $y_1' \in \CM_+'$ with $y_1' \leq \lambda$ such that 
	\begin{align*}
		\nu(x)= \inner{y_1' x \Omega_{\rho}}{ \Omega_{\rho}}_{\rho} \quad \text{for all } x \in \CM.
	\end{align*}

	For $y' \in \CM'$, we write 
	\[
	B_{n}(y') := \frac{1}{\mu(F_{n})}\int_{F_{n}} \alpha_{s}'(y') d \mu(s).
	\]
	By Lemma \ref{key lem1}, we obtain
	\begin{align*}
		\inner{B_{n}(y') x \Omega_{\rho}}{ \Omega_{\rho}}_{\rho} = \inner{y' A_{n}(x) \Omega_{\rho}}{ \Omega_{\rho}}_{\rho}, \quad \text{for all } x \in \CM,  y' \in \CM',
	\end{align*}
	here 
	\[
	A_{n}(x):= \frac{1}{\mu(F_{n})}\int_{F_{n}} \alpha_{s^{-1}}(x) d \mu(s).
	\] 
	
	\noindent For all $n \in \N$ and $x \in \CM_+$, we obtain 
	\begin{align*}
		\abs{(A_{n}(\nu_k) - \bar{\nu})(x)}
		&= \abs{(\nu_k - \bar{\nu})(A_{n}(x))} \\
		&= \abs{(\nu - A_{k}(\nu))(A_{n}(x))}  \quad \text{(since } \nu_k := \nu - A_{k}(\nu ) + \bar{\nu}) \\
		&= \abs{\nu(A_{n}(x)) - \nu(A_{k}(A_{n}(x)))}  \quad \text{(as } B_{k}(\nu)(\cdot) = \nu(B_{k}(\cdot)))\\
		&= \abs{\inner{y_1' A_{n}(x) \Omega_{\rho} }{\Omega_{\rho}}_{\rho} - \inner{y_1' A_{k}(A_{n}(x))\Omega_{\rho}}{\Omega_{\rho}}_{\rho}} \\
		&\quad \text{(since } \nu(\cdot) = \inner{y_1' (\cdot)\Omega_{\rho}}{ \Omega_{\rho}}_{\rho}) \\
		&= \abs{\inner{(y_1'- B_{k}(y_1'))A_{n}(x) \Omega_{\rho}}{\Omega_{\rho}}_{\rho}} \\
		&= \abs{\inner{B_{n}(y_1'- B_{k}(y_1'))x \Omega_{\rho}}{\Omega_{\rho}}_{\rho}}, \text{ (by Lemma } \ref{key lem1}) \\
		&\leq \norm{B_{n}(y_1'- B_{k}(y_1'))}_{\rho} \rho(x).
	\end{align*}
	
	\noindent Moreover, since Lemma \ref{dense set-lem} ensures that $\lim_{n \to \infty} \norm{B_{n}(y_1'- B_{k}(y_1'))}_{\rho}=0$, we conclude that
	\begin{align*}
		\lim_{n \to \infty} \sup_{x \in \CM_+, x \neq 0} \frac{\abs{(A_{n}(\nu_k)- \bar{\nu}_k)(x)}}{\rho(x)}
		&= \lim_{n \to \infty} \sup_{x \in \CM_+, x \neq 0} \frac{\abs{(A_{n}(\nu_k)- \bar{\nu})(x)}}{\rho(x)} \\
		&= 0.
	\end{align*}
	This completes the proof.		
\end{proof}

%%%%
\begin{rem}\label{dense set-rem1}
\begin{enumerate}
	\item From Proposition \ref{dense set}, we begin by observing that for every $p \in \CP(\CM)$ and $\nu \in \CW$, one has
	\begin{align*}
		\lim_{n \to \infty} 
		\sup_{\substack{x \in p\CM_+p \\ x \neq 0}}
		\frac{\abs{(A_n(\nu)-\bar{\nu})(x)}}{\rho(x)} = 0 .
	\end{align*}

	\item Let $\CM$ be a von Neumann algebra equipped with a faithful normal tracial state $\tau$, and let $\nu \in \CW$. Since $\rho \in \CM_{*+}$, there exists a unique element $X \in L^1(\CM,\tau)_+$ such that
	\[
		\rho(x) = \tau(Xx), \qquad \text{for all } x \in \CM .
	\]
	
	For each $s > 0$, define the projection $q_s := \chi_{(1/s,\, s)}(X) \in \CM$. Note that
	\[
		\tau(1 - q_s) \xrightarrow[s \to \infty]{} 0 .
	\]
	Consequently, for any $\varepsilon > 0$, there exists $s_0 > 0$ such that $\tau(1 - q_{s_0}) < \varepsilon$. Moreover, we have the operator inequality $X q_{s_0} \leq s_0 q_{s_0}$. Hence, for every nonzero element $x \in q_{s_0}\CM_+ q_{s_0}$,
	\begin{align*}
		\frac{\rho(x)}{\tau(x)} 
		&= \frac{\tau(Xx)}{\tau(x)} 
		= \frac{\tau(X q_{s_0} x)}{\tau(x)} \\
		&\leq \frac{\tau(s_0 x)}{\tau(x)} 
		= s_0 .
	\end{align*}
	
	\noindent It therefore follows from Proposition~\ref{dense set} that
	\begin{align*}
		\lim_{n \to \infty} 
		\sup_{\substack{x \in p\CM_+p \\ x \neq 0}}
		\frac{\abs{(A_n(\nu)-\bar{\nu})(x)}}{\tau(x)}
		&\leq s_0
		\lim_{n \to \infty} 
		\sup_{\substack{x \in p\CM_+p \\ x \neq 0}}
		\frac{\abs{(A_n(\nu)-\bar{\nu})(x)}}{\rho(x)} \\
		&= 0 .
	\end{align*}
\end{enumerate}
\end{rem}

		%\item Suppose the group $G$ is unimodular and amenable. Then,we may consider a symmetric F\o lner sequence $(F_n)$ in $G$, and by \cite[Theorem A]{cadilhac2022noncommutative}, further obtain a symmetric admissible filtered subsequence of $(F_n)$. So we may assume the F{\o}lner sequence $(F_n)$ to be symmetric, filtered and admissible. \\ Now observe that
%		\begin{align}\label{symm avg}
%				A_{n}(\nu):= \frac{1}{\mu(F_{n})} \int_{F_{n}} \alpha_s(\nu) d \mu(s) = \frac{1}{\mu(F_{n})} \int_{F_{n}} \alpha_{s^{-1}}(\nu) d \mu(s) =: B_{n}(\nu), \quad \forall \nu \in \CM_*.
%			\end{align}

\noindent Throughout the remainder of this section, we assume that the von Neumann algebra $\CM$ is endowed with a faithful normal tracial state $\tau$.

\noindent The next result establishes a suitable Banach principle, which plays a central role in deriving ergodic convergence results. To this end, we first introduce the notion of bilateral almost uniform (b.a.u.) convergence for sequences in the predual $\CM_*$.

\begin{defn}\label{bau conv- functionals}
	Let $(\nu_n)_{n \in \N}$ be a sequence in $\CM_*$ and let $\nu \in \CM_*$. We say that $\nu_n$ converges to $\nu$ in the bilateral almost uniform (b.a.u.) topology if, for every $\varepsilon > 0$, there exists a projection $p \in \CP(\CM)$ with $\tau(1-p) < \varepsilon$ such that
	\begin{align*}
		\lim_{n \to \infty} 
		\sup_{\substack{x \in p\CM_+p \\ x \neq 0}}
		\frac{\abs{(\nu_n - \nu)(x)}}{\tau(x)} = 0 .
	\end{align*}
\end{defn}

\begin{rem}\label{equiv of bau convg}
	Let $(\nu_n)_{n \in \N} \subset \CM_*$ and $\nu \in \CM_*$. By Proposition~\ref{predual prop}, there exist a sequence $(X_n)_{n \in \N}$ and an element $X \in L^1(\CM,\tau)$ such that $\nu_n(x) = \tau(X_n x)$ and $\nu(x) = \tau(Xx)$ for all $x \in \CM$.

	We claim that $(\nu_n)$ converges to $\nu$ in the b.a.u. topology, in the sense of Definition~\ref{bau conv- functionals}, if and only if the corresponding sequence $(X_n)$ converges to $X$ bilaterally almost uniformly in $L^1(\CM,\tau)$, as in Definition~\ref{bau conv defn}. Indeed, assume first that $\nu_n \to \nu$ in b.a.u. Then, for every $\varepsilon > 0$, there exists a projection $e \in \CP(\CM)$ with $\tau(1-e) < \varepsilon$ such that for every $\delta > 0$, one can find $n_0 \in \N$ satisfying, for all $n \geq n_0$,
	\begin{align*}
		\abs{(\nu_n - \nu)(exe)} 
		&\leq \delta \, \tau(exe), \qquad \forall x \in \CM_+ \setminus \{0\}.
	\end{align*}
	Using the trace representation, this inequality is equivalent to
	\begin{align*}
		\abs{\tau\!\left((X_n - X)(exe)\right)}
		&\leq \delta \, \tau(exe), \qquad \forall x \in \CM_+ \setminus \{0\},
	\end{align*}
	which, by the faithfulness of $\tau$, is in turn equivalent to
	\[
		\norm{e(X_n - X)e} \leq \delta .
	\]
	Consequently, $X_n \to X$ in the b.a.u. sense. The converse implication follows by reversing the above arguments.
\end{rem}

\noindent The following proposition shows that the predual $\CM_*$ is complete with respect to the b.a.u. topology.

\begin{prop}\label{complete wrt bau-functionals}
	Every b.a.u.-Cauchy sequence $(\nu_n)_{n \in \N}$ in $\CM_*$ converges in the b.a.u. topology to some $\nu \in \CM_*$.
\end{prop}

\begin{proof}
	The assertion follows directly from Theorem~\ref{complete wrt bau} together with Remark~\ref{equiv of bau convg}.
\end{proof}
%%%%%%

   %%%%%%

\noindent We now prove the Banach principle.
    	\begin{thm}[Banach principle]\label{Banach principle}
		Suppose $G$ is a \textit{l.c.s.c.} amenable group with admissible F{\o}lner sequence $(F_{n})_{n \in \mathbb{N}}$, and let $\CM$ be a finite von Neumann algebra equipped with a faithful normal tracial state $\tau$ and a faithful normal state $\rho$. 
		
		Let $(\CM, G, \alpha)$ be a non-commutative dynamical system such that $\rho \circ \alpha_g = \rho$ for all $g \in G$, i.e., $(\CM, G, \alpha, \rho)$ forms a kernel. Then, the following set 
		\begin{align*}
			\CC = \{ \phi \in \CM_{*s} \mid A_n(\phi) \text{ converges in b.a.u.} \}
		\end{align*}
		is closed in $\CM_{*s}$.
	\end{thm}

    	\begin{proof}
		Let $(\phi_n)$ be a sequence in $\mathcal{C}$, converging to $\phi \in \mathcal{M}_{*s}$ in norm, i.e.,
		$$ \lim_{n \to \infty} \| \phi_n - \phi \| = 0. $$
		Then, we need to show that $\phi \in \mathcal{C}$. 
		
		First, we note that since $\rho \in \mathcal{M}_{*+}$, there exists a unique $X \in L^1(\mathcal{M}, \tau)_+$ such that $\rho(x) = \tau(Xx)$ for all $x \in \mathcal{M}$. Then, for any $s > 0$, consider the projection $q_s := 1_{(1/s, s)}(X) \in \mathcal{M}$. Observe that $(1 - q_s) \xrightarrow{s \to \infty} 0$ in the strong operator topology (SOT). 
		
		Let $\theta > 0$ be arbitrary. Therefore, there exists an $s_0 > 0$ such that $\tau(1 - q_{s_0}) < \frac{\theta}{2}$. Further, this implies $X q_{s_0} \leq s_0 q_{s_0}$. Thus, for all $0 \neq x \in q_{s_0} \mathcal{M}_+ q_{s_0}$, we have 
		\begin{align}\label{rho-tau}
			\begin{split}
				\frac{\rho(x)}{\tau(x)} = \frac{\tau(Xx)}{\tau(x)}&= \frac{\tau(X q_{s_0} x)}{\tau(x)} \quad \text{ (since $q_{s_0} x = x$)}\\
				&\leq \frac{\tau(s_0 x)}{\tau(x)} = s_0.
			\end{split} 
		\end{align}
		
		Since $ \lim_{n \to \infty} \| \phi - \phi_n \| = 0 $, possibly passing through a subsequence, we can assume that 
		\begin{align*}\label{coneq-3}
			\|\phi - \phi_j\| < \frac{1}{4^{s_j}} \text{ for all } j \in \mathbb{N}.
		\end{align*}
		
		Furthermore, note that $\phi - \phi_j \in \mathcal{M}_{*s}$ for all $j \in \mathbb{N}$. Now, using Corollary \ref{maximal ineq for averages-cor} for $n = N$, $\upsilon = \frac{1}{2^i}$, we obtain positive reals $({\gamma_j(N)})_{N \in \mathbb{N}}$ and projections $ q_{N,j}^{i} := q^{ \frac{1}{2^i}}_{ N, j} $ in $\mathcal{M}$ for $ j , N \in \mathbb{N}$ such that:
		\begin{enumerate}\setlength\itemsep{.7em}
			\item $ \displaystyle \lim_{j \to \infty} \gamma_j(N) = 0$, uniformly in $N \in \mathbb{N}$,   
			\item $ \rho(1- q^{i}_{N,j}) \leq ( 1+ \frac{1}{2^i})\gamma_j(N) + \frac{1}{2^i} $, and
			\item for all $j \in \mathbb{N}$, 
			\[
			A_n(\phi - \phi_j)(x) \leq  \frac{cs_j}{2^{s_j}}  \rho(x) \text{ for all } x \in q^{i}_{N,j} \mathcal{M}_+ q^{i}_{N, j} \text{ and } n \in [N].
			\]
		\end{enumerate}
        Now denote $q_{N,j} = q^{j}_{N,j}$ \\
        We have,
        \begin{align*}
        \rho(1-q_{N,j}) \leq (1+ \frac{1}{2^j})\gamma_j(N) + \frac{1}{2^j}
        \end{align*} for all $N \in \mathbb{N}$.
		Now, it immediately follows that
		\[
		\displaystyle \lim_{j \to \infty}  \rho(1- q_{N, j}) = 0, \text{ uniformly in }  N \in \mathbb{N}. 
		\]
		We use Lemma \ref{uniform conv lem}. Let $\CM_{1}=L^{\infty}(G)$ and $\CM_{2}=\CM$, take $S$ to be the set of natural numbers $\mathbb{N}$, let  $x_{j}(n)=q_{j,n} \text{ and } p_{j}(n)=\chi_{F_{n}}  \forall n \in \mathbb{N}$. Note that $f_{j}(N)=\frac{p_{j}(N) \otimes 1-q_{j,N}}{\mu(F_{N})},\forall N \in \mathbb{N}$. We have,
        \begin{align*}
        &\lim_{j \rightarrow \infty}\frac{(\varphi \otimes \rho)(\chi_{F_{N}}\otimes(1-q_{N,j}))}{\varphi(\chi_{F_{N}})} = 0 \\
        \implies & \lim_{j \rightarrow \infty}\frac{(\varphi \otimes \tau)(\chi_{F_{N}}\otimes(1-q_{N,j}))}{\varphi(\chi_{F_{N})}} = 0 \text{ by Lemma 3.10 } \\
        \implies & \lim_{j \rightarrow \infty}\frac{\varphi(\chi_{F_{N}})\tau(1-q_{N,j}))}{\varphi(\chi_{F_{N}})} = 0 \\
        \implies & \lim_{j \rightarrow \infty}\tau(1-q_{N,j})=0, \text{ uniformly in N}\in \mathbb{N}.
        \end{align*}

		\noindent Therefore, there exists a subsequence $(j_l)_{l \in \mathbb{N}}$ such that  
		\begin{align*}
			\tau(1 - q_{N, j_l}) < \frac{\theta}{2^{l+2}},  \quad \text{for all } N \in \mathbb{N}.
		\end{align*}

		Since $\phi_j \in \mathcal{C}$ for all $j \in \mathbb{N}$, there exists a sequence of projections $(f_j)$ in $\mathcal{P}(\mathcal{M})$ and a strictly increasing sequence $(N_j) \subseteq \mathbb{N}$ such that  
		\begin{align*}
			\begin{split}
				(1)\quad &\tau(1 - f_j)  < \frac{\theta}{2^{j+2}},  \quad \text{for all } j \in \mathbb{N}, \\
				(2)\quad &\sup_{x \in f_j M_+ f_j, x \neq 0} \frac{\left| \big(A_n(\phi_j) - A_m(\phi_j) \big)(x) \right|}{\tau(x)} < \frac{1}{2^j},  \quad \text{for all } n, m \geq N_j.
			\end{split}
		\end{align*}
		
		\noindent Now, for all $l \in \mathbb{N}$, we define $q_l := q_{N_{j_{l+1}}, j_l}$. Then, we note that  
		\begin{align*}
			\begin{split}
				(1)\quad &\tau(1 - q_l) \leq \frac{\theta}{2^{l+2}}, \\
				(2)\quad &A_n(\phi - \phi_{j_l}) (q_l x q_l) \leq   \frac{cs_{j_{l}}}{2^{s_{j_l}}} \rho(q_l x q_l),  
				\quad \text{for all }n \in [N_{j_{l +1}}].
			\end{split}
		\end{align*}

		\noindent Moreover, for the subsequence $(j_l)$, we have  
		\begin{align*}
			\tau(1 - f_{j_l}) \leq \frac{\theta}{2^{j_l+2}}.
		\end{align*}
		
		\noindent Finally, consider the projection  
		\begin{align*}
			e= \left( \underset{l \in \mathbb{N}}{\bigwedge} q_l \right) \bigwedge \left( \underset{l \in \mathbb{N}}{\bigwedge} f_{j_l} \right) \bigwedge q_{s_0}.
		\end{align*}
		Then, we have $\tau(1 - e) < \theta$ and, for all $0 \neq x \in e \mathcal{M}_+ e$, assuming $N_{j_l} \leq n, m \leq N_{j_{l+1}}$, we obtain  
		\begin{align*}
			&\frac{\left| \big(A_n( \phi) - A_m(\phi) \big) (x) \right|}{\tau(x)}\\
			=&  \frac{\left| \big(A_n( \phi) - A_n(\phi_{j_l}) + A_n(\phi_{j_l}) - A_m(\phi_{j_l}) + A_m(\phi_{j_l}) - A_m(\phi) \big) (x) \right|}{\tau(x)}\\
			\leq &  \frac{\left| \big(A_n( \phi) - A_n(\phi_{j_l}) \big)(x)  \right| }{\tau(x) }
			+ \frac{\left| \big(A_n(\phi_{j_l}) - A_m(\phi_{j_l}) \big) (x) \right|}{\tau(x)} 
			+ \frac{\left| \big( A_m(\phi_{j_l}) - A_m(\phi) \big) (x) \right|}{\tau(x)}\\
			\leq &  \frac{\left| \big(A_n( \phi - \phi_{j_l}) \big)(x)  \right| }{\tau(x) }
			+ \frac{\left| \big(A_n(\phi_{j_l}) - A_m(\phi_{j_l}) \big) (x) \right|}{\tau(x)} 
			+ \frac{\left| \big( A_m(\phi_{j_l} - \phi) \big) (x) \right|}{\tau(x)}\\
			\leq &  \left(\frac{\left| \big(A_n( \phi - \phi_{j_l}) \big)(x)  \right| }{\rho(x) } 
			+ \frac{\left| \big( A_m(\phi_{j_l} - \phi) \big) (x) \right|}{\rho(x)} \right) 
			\frac{\rho(x)}{\tau(x)} 
			+ \frac{\left| \big(A_n(\phi_{j_l}) - A_m(\phi_{j_l}) \big) (x) \right|}{\tau(x)}\\
			\leq&  \left(\frac{2cs_{j_{l}}}{2^{s_{j_{l}}}}\right)s_{0} + \frac{1}{2^{s_{j_{l}}}} = \frac{(2cs_{j_{l}}s_{0} +1)} {2^{s_{j_{l}}}}, 
			\quad (\text{by Eq. } \ref{rho-tau}). 
		\end{align*}	
		
		\noindent Since $j_l \to \infty$ as $m,n \to \infty$, it follows that  
		\begin{align*}
			\lim_{m,n \to \infty} \sup_{x \in e M_+ e, x \neq 0} \frac{\left| \big(A_n( \phi) - A_m(\phi) \big) (x) \right|}{\tau(x)} = 0.
		\end{align*}
		
		\noindent Therefore, $(A_n(\phi))_{n \in \mathbb{N}}$ is Cauchy in b.a.u. Thus, by Proposition \ref{complete wrt bau-functionals}, it follows that $(A_n(\phi))_{n \in \mathbb{N}}$ is convergent in b.a.u., which implies that the set $\mathcal{C}$ is closed in $\mathcal{M}_{*s}$.		
	\end{proof}
\noindent The following results require that $G$ is unimodular, due to our use of proposition \ref{dense set}(ii).We conclude with the following ergodic theorem.
\begin{thm}\label{erg thm-1}
	Let $(\mathcal{M}, G, \alpha, \rho)$ be a kernel, and let $\tau$ be a faithful normal tracial state on $\mathcal{M}$. For any $\phi \in \mathcal{M}_*$, there exists an element $\bar{\phi} \in \mathcal{M}_*$ such that, for every $\varepsilon > 0$, one can find a projection $e \in \mathcal{M}$ satisfying $\tau(1-e) < \varepsilon$ and
	\begin{align*}
		\lim_{n \to \infty}
		\sup_{\substack{x \in e\mathcal{M}_+e \\ x \neq 0}}
		\frac{\abs{(A_n(\phi) - \bar{\phi})(x)}}{\tau(x)} = 0 .
	\end{align*}
\end{thm}

\begin{proof}
	The assertion is an immediate consequence of Proposition~\ref{dense set} and Theorem~\ref{Banach principle}, combined with Remark~\ref{dense set-rem1}.
\end{proof}

\noindent Let $(\mathcal{M}, G, \alpha, \rho)$ be a kernel, and let $\tau$ be a faithful normal tracial state on $\mathcal{M}$. We next reformulate Theorem~\ref{erg thm-1} in the framework of the noncommutative $L^1$-space $L^1(\mathcal{M},\tau)$. For a given element $X \in L^1(\mathcal{M},\tau)$, recall from the Eq.\ref{predual trans} the associated ergodic averages
\begin{align*}
	A_n(X)
	:= \frac{1}{\mu(F_n)} \int_{F_n} \hat{\alpha}_{s^{-1}}(X) \, d\mu(s).
\end{align*}

%\noindent As in \eqref{symm avg}, these averages may equivalently be written in the symmetric form
%\begin{align*}
%	A_n(X)
%	= \frac{1}{\mu(F_n)} \int_{F_n} \hat{\alpha}_s(X) \, d\mu(s).
%\end{align*}

\begin{thm}\label{erg thm-2}
	Let $(\mathcal{M}, G, \alpha, \rho)$ be a kernel, and let $X \in L^1(\mathcal{M},\tau)$. Then there exists $\bar{X} \in L^1(\mathcal{M},\tau)$ such that the sequence $(A_n(X))$ converges bilaterally almost uniformly to $\bar{X}$.
\end{thm}

\begin{proof}
	The conclusion follows directly from Theorem~\ref{erg thm-1} together with Remark~\ref{equiv of bau convg}.
\end{proof}

\begin{rem}
   In view of \Cref{erg thm-2}, we observe that a non-commutative pointwise ergodic theorem, analogous to \cite[Theorem~4.17]{Bik-Dip-neveu}, can be obtained for ergodic averages taken along the F\o lner sequences associated with an unimodular amenable group action considered in the dominance result of \cite[Theorem~3.1]{chakraborty2025ergodic}.

\end{rem}

\section{\textbf{Stochastic Ergodic Theorem}}

This section brings together the results developed so far in order to derive a stochastic ergodic theorem. Throughout, we  assume that $\mathcal{M} \subseteq \mathcal{B}(\mathcal{H})$ is a von Neumann algebra endowed with a faithful normal tracial state $\tau$, and that we work in the associated noncommutative space $L^1(\mathcal{M},\tau)$. Moreover, we assume that $G$ is a unimodular amenable group admitting a symmetric admissible Følner sequence $(F_n)$.

We consider a noncommutative dynamical system $(\mathcal{M}, G, \alpha)$. As before, for each $X \in L^1(\mathcal{M},\tau)$ we define the corresponding ergodic averages by
\begin{align*}
	A_n(X)
	:= \frac{1}{\mu(F_n)} \int_{F_n} \alpha_s^*(\nu)\, d\mu(s)
	= \frac{1}{\mu(F_n)} \int_{F_n} \hat{\alpha}_{s^{-1}}(X)\, d\mu(s).
\end{align*}

Relying on the preceding results, we now establish a stochastic ergodic theorem for the averages $A_n(\cdot)$. To this end, we first recall the following decomposition theorem from \cite{Bik-Dip-neveu}.

\begin{thm}[Neveu Decomposition]\label{neveu decomposition}
	Let $(\CM, G, \alpha)$ be a noncommutative dynamical system, and let $\tau$ be a faithful normal tracial state on $\CM$. Then there exist projections $e_1, e_2 \in \CP(\CM)$ satisfying $e_1 + e_2 = 1$ and $e_1 e_2 = 0$, with the following properties:
	\begin{itemize}
		\item[(i)] There exists a $G$-invariant normal state $\rho$ on $\CM$ whose support projection is $s(\rho) = e_1$.
		\item[(ii)] There exists an element $x_0 \in \CM$ with support $s(x_0) = e_2$ such that
		\[
			\frac{1}{\mu(F_n)} \int_{F_n} \alpha_s(x_0)\, d\mu(s)
			\xrightarrow[n \to \infty]{} 0
			\quad \text{in } \norm{\cdot}.
		\]
	\end{itemize}
	Moreover, the projections $s(\rho)$ and $s(x_0)$ are uniquely determined.
\end{thm}

We next record several basic consequences of this decomposition, following \cite{Bik-Dip-neveu}.

\begin{enumerate}
	\item For every $g \in G$ and $i = 1,2$, the projections $e_i$ are $\alpha$-invariant, that is, $\alpha_g(e_i) = e_i$.
	\item The restriction of $\alpha$ to the reduced von Neumann algebra $\CM_{e_i} := e_i \CM e_i$ induces an action by automorphisms. Consequently, for each $i=1,2$, the triple $(\CM_{e_i}, G, \alpha)$ forms a noncommutative dynamical system.
	\item For $i = 1,2$ and all $g \in G$, the induced predual action $\hat{\alpha}_g$ on $L^1(\CM,\tau)$ satisfies
	\[
		\hat{\alpha}_g(e_i X e_i)
		= e_i \hat{\alpha}_g(X) e_i,
		\qquad \forall X \in L^1(\CM,\tau).
	\]
	As a consequence, the ergodic averages obey
	\[
		e_i A_n(X) e_i = A_n(e_i X e_i),
		\qquad \forall X \in L^1(\CM,\tau),\ n \in \N,\ i=1,2.
	\]
\end{enumerate}

For $i=1,2$, we define the normalized traces $\tau_{e_i} := \frac{1}{\tau(e_i)}\, \tau|_{e_i \CM e_i}$. We are now in a position to state the main convergence result.

\begin{thm}\label{conv in e1-e2 corner}
	Let $(\CM, G, \alpha)$ be a noncommutative dynamical system, and let $\tau$ be a faithful normal tracial state on $\CM$. Let $e_1, e_2 \in \CM$ be the projections arising from the Neveu decomposition. Then the following assertions hold:
	\begin{enumerate}
		\item[(i)] For every $B \in L^1(\CM_{e_1}, \tau_{e_1})$, there exists $\bar{B} \in L^1(\CM_{e_1}, \tau_{e_1})$ such that $A_n(B)$ converges bilaterally almost uniformly to $\bar{B}$. In addition, $A_n(B)$ converges to $\bar{B}$ in measure.
		\item[(ii)] For every $B \in L^1(\CM_{e_2}, \tau_{e_2})$, the sequence $A_n(B)$ converges to $0$ in measure.
	\end{enumerate}
	Moreover, for any $X \in L^1(\CM,\tau)$, there exists an element $Z \in L^1(\CM,\tau)$ such that $A_n(X)$ converges to $Z$ in measure and satisfies
	\[
		e_1 Z e_1 = Z, \qquad e_2 Z e_2 = 0.
	\]
\end{thm}

\begin{proof}
	Given the Neveu decomposition (Theorem~\ref{neveu decomposition}) and the b.a.u.\ convergence result of Theorem~\ref{erg thm-2}, the argument proceeds verbatim as in \cite[Section~5]{Bik-Dip-neveu}.
\end{proof}

%\textbf{Acknowledgments:} The authors are grateful to the anonymous referee(s) for their careful reading of an earlier
%version of the manuscript and for several insightful comments and suggestions, which led to a better presentation of \Cref{maximal ineq for averages} and \Cref{Banach principle}.
%The second author is also thankful to Prof. Eric Ricard for helpful discussions on martingale convergence.

	\noindent\textbf{Acknowledgements}: 
	%D. Saha is supported by Council of Scientific and Industrial Research grant no. 09/1002(0029)/2017-EMR-I, Ministry of Human Resource Development, Government of India.
 P. Bikram acknowledges the support of the grant ANRF/ARGM/2025/001021/MTR, 
	Government of India.

%	\bibliographystyle{amsalpha}
%	\bibliography{ref}

\newcommand{\etalchar}[1]{$^{#1}$}
\providecommand{\bysame}{\leavevmode\hbox to3em{\hrulefill}\thinspace}
\providecommand{\MR}{\relax\ifhmode\unskip\space\fi MR }
% \MRhref is called by the amsart/book/proc definition of \MR.
\providecommand{\MRhref}[2]{%
	\href{http://www.ams.org/mathscinet-getitem?mr=#1}{#2}
}
\providecommand{\href}[2]{#2}

\end{document}